\documentclass[11pt]{article}

\usepackage[margin=1in]{geometry}
\usepackage{amsmath,amssymb,amsthm,mathtools}
\usepackage{array}
\usepackage{graphicx}
\usepackage{hyperref}

\newtheorem{theorem}{Theorem}
\newtheorem{lemma}[theorem]{Lemma}
\newtheorem{proposition}[theorem]{Proposition}
\newtheorem{corollary}[theorem]{Corollary}

\title{The Grimmer--Shu--Wang Certificate and the Drori--Teboulle Minimax Constant-Stepsize Bound for \(N\ge3\)}
\author{
Lixing Zhang\thanks{The Chinese University of Hong Kong, Shenzhen,
\texttt{122090748@link.cuhk.edu.cn}.}
}
\date{}

\begin{document}
\maketitle

\begin{abstract}
We prove, for every horizon $N\ge3$, the existence theorem for the strengthened low-rank
performance-estimation certificate proposed by Grimmer, Shu, and Wang for the
Drori--Teboulle constant-step gradient-descent bound.  For every $N\ge3$, let
$\rho_N\in(0,1)$ be determined by
\[
  \rho_N^{2N}(2N\rho_N+2N+1)=1 .
\]
We prove that the GSW certificate equations admit positive vectors
$a,b,c,d$ satisfying all residual equations.  The proof proceeds through a
reduced residual system in the variables $d$, a simplex existence argument for
a positive reduced zero, a terminal residual completion identity, and a
tail-square convolution argument proving the cumulative margins that force
$b>0$ and then $a>0$.  Consequently the GSW low-rank PEP certificate exists for
every $N\ge3$ and yields the Drori--Teboulle upper bound.  We also include the
one-dimensional quadratic and Huber lower-bound examples for
nonnegative steps, and the quadratic example also excludes negative constant
steps from being optimal.  Together these prove the Drori--Teboulle minimax
constant-stepsize value over all real constant steps for every $N\ge3$.
\end{abstract}

\section{Introduction}

Let $f:\mathbb{R}^p\to\mathbb{R}$ be convex and $1$-smooth.  Consider constant-step gradient
descent
\[
  x_{k+1}=x_k-\alpha \nabla f(x_k),\qquad k=0,\ldots,N-1 ,
\]
with a real constant step $\alpha\in\mathbb R$.
The performance-estimation framework of Drori and Teboulle
\cite{drori-teboulle} converts worst-case analysis of first-order methods into
finite-dimensional interpolation inequalities.  In the normalization
$\|x_0-x_\star\|^2\le 1$, the conjectured
optimal constant step is written
\[
  \alpha=1+\rho,\qquad
  \rho^{2N}(2N\rho+2N+1)=1,
\]
and the corresponding upper-bound value is
\[
  r=\frac{\rho^{2N}}{2}
   =\frac{1}{2(2N\alpha+1)}.
\]

Grimmer, Shu, and Wang \cite{grimmer-shu-wang} proposed a strengthened
low-rank dual certificate
whose existence would imply this upper bound.  The certificate is encoded by
positive coefficients $a,b,c,d$ satisfying explicit triangular systems and
residual identities.  The purpose of this paper is to prove that such
coefficients exist for every $N\ge3$ and to combine the resulting upper bound
with the matching one-dimensional lower examples derived below.

Classical oracle complexity and first-order complexity theory
\cite{nemirovski-yudin,polyak,nesterov-1983,nesterov-lectures} give
iteration-order guarantees and motivate the search for sharp constants.  PEP
turns that sharp-constant question into an optimization problem over finite
first-order data.  Smooth-convex interpolation, as developed for exact
worst-case analysis by Taylor, Hendrickx, and Glineur
\cite{taylor-hendrickx-glineur}, gives the exact finite characterization of
feasible function values, points, and gradients used by this dual viewpoint.
Optimized first-order methods \cite{kim-fessler} and IQC-based certificate
methods \cite{lessard-recht-packard} provide neighboring perspectives on
algorithm design and verification; here the focus is narrower, namely the
constant-step gradient method and the particular low-rank PEP dual structure
proposed by GSW.

This constant-stepsize certificate question also fits into the recent line
of work on optimized stepsize schedules and computer-assisted performance
estimation.  Grimmer, Shu, and Wang~\cite{grimmer-shu-wang-composing}
develop a composition theory for optimized gradient-descent stepsize
schedules, while Das Gupta, Van Parys, and Ryu~\cite{das-gupta-van-parys-ryu}
give a branch-and-bound PEP framework for constructing and certifying
optimal first-order methods.  The present paper is complementary: it gives
an explicit analytic proof for the constant-stepsize low-rank certificate
conjectured in~\cite{grimmer-shu-wang}.

The Drori--Teboulle constant-step problem is subtle because the best
step balances two different one-dimensional obstructions: the
quadratic example and the Huber example.  GSW strengthened the original
constant-step conjecture by specifying a structured low-rank family of
multipliers that should certify the matching upper bound for every $N\ge3$.  The
present paper proves the existence of that structured certificate for every $N\ge3$ and
then combines it with the explicit lower-bound constructions in
Section~\ref{sec:minimax}.

The strengthened GSW statement can be summarized as follows.  Rather than only
predicting the value of the constant-step PEP, it predicts a particular
low-rank dual multiplier pattern.  In the notation of this paper, the
correspondence is
\[
\resizebox{\textwidth}{!}{$
\begin{array}{c|c|c}
\text{GSW object} & \text{notation here} & \text{role in the proof}\\ \hline
\text{balanced root parameter} & \rho_N & \alpha_\star=1+\rho_N,\quad
  r_\star=\rho_N^{2N}/2\\
\text{low-rank multipliers} & a,b,c,d & \lambda_{\star j}=c_j,\;
  \lambda_{i,i+1}=a_i,\;\lambda_{i+1,i}=b_i,\;
  \lambda_{ij}=d_i c_j\ (i+2\le j)\\
\text{triangular certificate equations} &
  \text{Systems C, A, B} & \text{generate }c,a,b\text{ from }d\\
\text{residual equations} & \varepsilon_0=\cdots=\varepsilon_N=0 &
  \text{remaining dual coefficient equations}\\
\text{positivity requirement} & a,b,c,d>0 &
  \text{nonnegative interpolation multipliers}
\end{array}
$}
\]
Thus the certificate-existence theorem proved below is the GSW strengthened
certificate assertion in this normalization.

\begin{proposition}[GSW strengthened certificate conjecture, normalized form]
\label{prop:gsw-normalized-form}
Fix $N\ge3$ and let $\rho_N$ be the balanced root above.  In the normalization
used here, the strengthened GSW certificate assertion is the following:
there exist $d_0,\ldots,d_{N-2}>0$ such that Systems C, A, and B generate
coefficients $c_0,\ldots,c_N$, $a_0,\ldots,a_{N-1}$, and
$b_0,\ldots,b_{N-2}$ with
\[
  a_i>0,\qquad b_i>0,\qquad c_i>0,
  \qquad
  \varepsilon_0=\cdots=\varepsilon_N=0 .
\]
With the low-rank multiplier pattern
\[
  \lambda_{\star j}=c_j,\qquad
  \lambda_{i,i+1}=a_i,\qquad
  \lambda_{i+1,i}=b_i,\qquad
  \lambda_{ij}=d_ic_j\quad(i+2\le j),
\]
this is exactly the positive low-rank PEP certificate whose existence is
predicted by the GSW strengthened assertion, expressed in the present
notation.
\end{proposition}

\begin{proof}
The preceding dictionary identifies each object in the GSW multiplier ansatz
with the variables used in this paper.  Systems C, A, B, and D below are the
coefficient equations obtained by substituting this multiplier pattern into
the smooth-convex interpolation dual identity of
Proposition~\ref{prop:cert-rate}.  Positivity of $a,b,c,d$ is precisely
nonnegativity of all displayed interpolation multipliers.
\end{proof}

\paragraph{Relation to the original GSW assertion and to real steps.}
The strengthened GSW assertion is a certificate assertion for the
nonnegative constant-step problem, and the certificate constructed here is
used at the positive step \(\alpha_\star=1+\rho_N\).  The final extension of
the minimax statement from nonnegative constant steps to all real constant
steps is separate: Lemma~\ref{lem:negative-steps} uses the one-dimensional
quadratic lower example to exclude every negative step.  Thus no part of the
GSW low-rank ansatz is being interpreted as a certificate for negative
steps; the all-real minimax conclusion is obtained by combining the positive
GSW certificate with this additional negative-step exclusion.

The proof is organized by the following dependency chain:
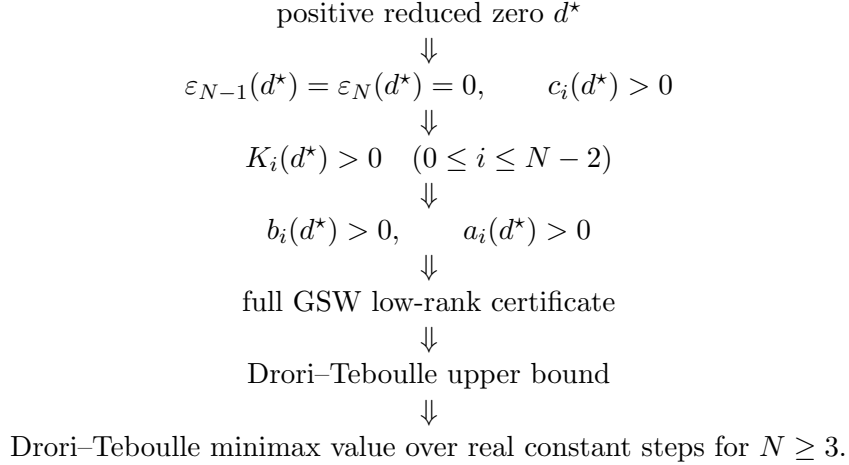
\[
\begin{array}{c}
\text{positive reduced zero }d^\star\\
\Downarrow\\
\varepsilon_{N-1}(d^\star)=\varepsilon_N(d^\star)=0,\qquad c_i(d^\star)>0\\
\Downarrow\\
K_i(d^\star)>0\quad(0\le i\le N-2)\\
\Downarrow\\
b_i(d^\star)>0,\qquad a_i(d^\star)>0\\
\Downarrow\\
\text{full GSW low-rank certificate}\\
\Downarrow\\
\text{Drori--Teboulle upper bound}\\
\Downarrow\\
\text{Drori--Teboulle minimax value over real constant steps for }N\ge3.
\end{array}
\]
In words, the proof first finds a zero of the reduced residuals
\(\varepsilon_0,\ldots,\varepsilon_{N-2}\).  The direct conservation
identities and the terminal quadratic identity then complete the two terminal
residuals.  Only after terminal completion do we use the tail-square law; the
tail-square law feeds the cumulative margins \(K_i\), which force positivity of
\(b\) and then \(a\).  The positive certificate gives the upper bound, and the
quadratic and Huber examples give the matching lower bound.
The central certificate-completion point is cumulative rather than
pointwise: the needed positivity is not positivity of the increments of $b$,
but positivity of the margins $K_i$ defined in Section~\ref{sec:coeff-pos}.

\paragraph{What has to be proved.}
The dual certificate identity reduces the upper bound to a finite algebraic
task.  Systems C, A, and B generate the multipliers \(c,a,b\) from the reduced
unknown \(d\), while System D records the remaining coefficient residuals
\(\varepsilon_i\).  Thus the proof has three separate responsibilities:
first find a positive reduced zero of
\(\varepsilon_0,\ldots,\varepsilon_{N-2}\); then prove, without introducing a
new zero, that the two terminal residuals vanish at the same \(d\); and only
afterward prove that the resulting multipliers are positive.  The lower-bound
section is independent of this certificate construction.

\section{The GSW certificate system}
\label{sec:system}
The normalization used in this section follows the
Drori--Teboulle performance-estimation formulation for constant-step
gradient descent and the strengthened low-rank certificate ansatz of
Grimmer, Shu, and Wang~\cite{drori-teboulle,grimmer-shu-wang}.

Fix $N\ge3$ and let $\rho=\rho_N$ be the unique root in $(0,1)$ of
\begin{equation}
  \rho^{2N}(2N\rho+2N+1)=1.
  \label{eq:rho-root}
\end{equation}
Set
\[
  \alpha=1+\rho,\qquad q=\rho^N,\qquad R=q^2=\rho^{2N},
  \qquad r=\frac{R}{2}.
\]
The equation \eqref{eq:rho-root} is equivalent to
\[
  R=\frac{1}{2N\alpha+1}.
\]
In particular $r=1/(2(2N\alpha+1))$.

\begin{lemma}[Location of the GSW root]
\label{lem:rho-root-location}
For every $N\ge3$, equation \eqref{eq:rho-root} has a unique root
$\rho_N\in(0,1)$, and this root satisfies $1/2<\rho_N<1$.
\end{lemma}

\begin{proof}
The function
\[
  \phi_N(\rho)=\rho^{2N}(2N\rho+2N+1)
\]
is strictly increasing on $\rho>0$, since every factor and its derivative are
positive there.  Also $\phi_N(0)=0$ and $\phi_N(1)=4N+1>1$, so there is a
unique root in $(0,1)$.  Finally
\[
  \phi_N(1/2)
  =2^{-2N}(3N+1)<1
\]
for $N\ge3$, whence the root is larger than $1/2$.
\end{proof}

Let $d=(d_0,\ldots,d_{N-2})$.  Empty sums below are understood to be zero. The equations below are the normalized coefficient equations obtained
from the strengthened GSW low-rank multiplier pattern
\(\lambda_{\star j}=c_j\), \(\lambda_{i,i+1}=a_i\),
\(\lambda_{i+1,i}=b_i\), and \(\lambda_{ij}=d_i c_j\) for
\(i+2\le j\)~\cite{grimmer-shu-wang}.
The system is triangular except for the residual equations:
\[
  d\Longrightarrow c(d)\Longrightarrow a(d),b(d)
  \Longrightarrow \varepsilon(d).
\]
The point of this parametrization is that the only genuinely free
variables are the entries of \(d\).  Once \(d\) is fixed, System C fixes
the star-row multipliers \(c\), and Systems A and B then solve
backwards for the adjacent multipliers \(a\) and \(b\).  The remaining
quantities \(\varepsilon_i\) are not additional definitions of
multipliers; they are the uncancelled coefficients in the dual identity.
Thus the certificate problem is reduced to finding a positive \(d\) for
which all these residual coefficients vanish.
System C defines $c_0,\ldots,c_N$ by
\begin{align}
  c_i
    &=R\left(\alpha\sum_{k=0}^i d_k-d_i+\alpha\right),
      &&0\le i\le N-2,                                      \label{eq:C1}\\
  c_{N-1}
    &=R\left(1+\sum_{k=0}^{N-2}d_k+\frac{\rho}{q}\right),     \label{eq:C2}\\
  c_N&=q .                                                    \label{eq:C3}
\end{align}
Systems A and B then define $a_0,\ldots,a_{N-1}$ and
$b_0,\ldots,b_{N-2}$ by backward recursion:
\begin{align}
  a_{N-1}&=1-c_N\left(1+\sum_{k=0}^{N-2}d_k\right),            \label{eq:Aterm}\\
  a_{N-2}
   &=
   \frac{c_{N-1}^2/R+c_{N-2}c_{N-1}/R-a_{N-1}
       -(1+\alpha)c_{N-1}\left(1+\sum_{k=0}^{N-3}d_k\right)}
        {\alpha},                                             \label{eq:Apen}\\
  b_{N-2}
   &=
   \frac{\rho c_{N-1}^2/R-c_{N-2}c_{N-1}/R-\rho a_{N-1}
       +c_{N-1}\left(1+\sum_{k=0}^{N-3}d_k\right)}
        {\alpha}.                                             \label{eq:Bpen}
\end{align}
For $i=N-3,\ldots,0$,
\begin{align}
  a_i
   &=
   \frac{
    c_{i+1}^2/R+c_i c_{i+1}/R-a_{i+1}
    -(1+\alpha)c_{i+1}\left(1+\sum_{k=0}^{i-1}d_k\right)
    -d_{i+1}\sum_{j=i+3}^N c_j
    +(2\alpha-1)b_{i+1}}
   {\alpha},                                                  \label{eq:Arec}\\
  b_i
   &=
   \frac{
    \rho c_{i+1}^2/R-c_i c_{i+1}/R-\rho a_{i+1}
    +c_{i+1}\left(1+\sum_{k=0}^{i-1}d_k\right)
    -\rho d_{i+1}\sum_{j=i+3}^N c_j
    +\rho(2\alpha-1)b_{i+1}}
   {\alpha}.                                                  \label{eq:Brec}
\end{align}
We call the following residual equations System D.  They record the
remaining function-value and endpoint coefficients after Systems C, A,
and B have cancelled the prescribed point-gradient and Gram
coefficients.  Define \(\varepsilon_0,\ldots,\varepsilon_N\) by
\begin{align}
  \varepsilon_0&=
    a_0+d_0\sum_{j=2}^N c_j-b_0-c_0,                           \label{eq:D0}\\
  \varepsilon_i&=
    b_{i-1}+a_i+d_i\sum_{j=i+2}^N c_j
    -a_{i-1}-b_i
    -c_i\left(1+\sum_{k=0}^{i-2}d_k\right),
        &&1\le i\le N-2,                                      \label{eq:Dmid}\\
  \varepsilon_{N-1}&=
    b_{N-2}+a_{N-1}-a_{N-2}
    -c_{N-1}\left(1+\sum_{k=0}^{N-3}d_k\right),                 \label{eq:Dterm1}\\
  \varepsilon_N&=
    -c_0-a_0-d_0\sum_{j=2}^N c_j+(2\alpha-1)b_0+\frac{c_0^2}{R}.
                                                                        \label{eq:Dterm2}
\end{align}

\subsection{A small horizon illustration}

For $N=3$ the reduced unknown is $d=(d_0,d_1)$.  System C defines
$c_0,c_1,c_2,c_3$ from this pair, Systems A and B then define
$a_0,a_1,a_2$ and $b_0,b_1$ by the backward recursions, and System D gives
four residuals $\varepsilon_0,\varepsilon_1,\varepsilon_2,\varepsilon_3$.
The reduced system is
\[
  E_3(d)=(\varepsilon_0(d),\varepsilon_1(d))=0 .
\]
The terminal completion in Section~\ref{sec:terminal} proves that this same
solution $d^\star$ also satisfies
$\varepsilon_2(d^\star)=\varepsilon_3(d^\star)=0$.  The dual weights then
have the pattern
\[
\begin{array}{c|c}
\text{weight family} & \text{nonzero weights for }N=3\\ \hline
\lambda_{\star j} & c_0,c_1,c_2,c_3\\
\lambda_{i,i+1} & a_0,a_1,a_2\\
\lambda_{i+1,i} & b_0,b_1\\
\lambda_{ij}=d_ic_j,\ i+2\le j & d_0c_2,d_0c_3,d_1c_3 .
\end{array}
\]
This finite case illustrates the general construction map
$d\mapsto c(d)\mapsto a(d),b(d)\mapsto\varepsilon(d)$ without introducing
any additional branch.

The purpose of the multipliers is to form a nonnegative weighted sum of
smooth-convex interpolation inequalities.  The star-row choice
$\lambda_{\star i}=c_i$ cancels the point-gradient coefficients against
$v=\sum_i c_i g_i$.  Systems C, A, and B then parametrize $c,a,b$ so that the
Gram coefficients cancel, while System D records the remaining function-value
and endpoint coefficients as residuals.
When the residuals vanish and the multipliers are positive, the weighted sum
becomes exactly the desired upper-bound certificate.

\begin{proposition}[GSW dual identity implies the PEP rate]
\label{prop:cert-rate}
Let $f$ be convex and $1$-smooth.  Run constant-step gradient descent
$x_{k+1}=x_k-\alpha\nabla f(x_k)$ for $N$ steps, and write
\[
  f_i=f(x_i),\qquad g_i=\nabla f(x_i),\qquad f_\star=f(x_\star),
  \qquad g_\star=0.
\]
For indices \(i,j \in \{\star,0,\ldots,N\}\) define the
smooth-convex interpolation slack
\[
Q_{ij}
=
f_i-f_j-\langle g_j,x_i-x_j\rangle
-\frac12\|g_i-g_j\|^2 .
\]
The smooth-convex interpolation theorem implies
\(Q_{ij}\ge0\) for all \(i,j\); see
Taylor, Hendrickx, and Glineur~\cite{taylor-hendrickx-glineur}.  Given coefficients $a,b,c,d$, define
dual weights by
\[
  \lambda_{\star,j}=c_j,\qquad 0\le j\le N,
\]
\[
  \lambda_{i,i+1}=a_i,\qquad 0\le i\le N-1,
\]
\[
  \lambda_{i+1,i}=b_i,\qquad 0\le i\le N-2,
\]
and
\[
  \lambda_{i,j}=d_i c_j,\qquad 0\le i\le N-2,\quad i+2\le j\le N,
\]
with all other $\lambda_{ij}$ equal to zero.

If the coefficients defined above are positive and satisfy
\[
  \varepsilon_0=\cdots=\varepsilon_N=0,
\]
then, for constant-step gradient descent with step $\alpha=1+\rho_N$,
\[
  f(x_N)-f_\star\le
  \frac{\rho_N^{2N}}{2}\|x_0-x_\star\|^2
  =
  \frac{1}{2(2N\alpha+1)}\|x_0-x_\star\|^2 .
\]
\end{proposition}

\begin{proof}
Put
\[
  y_i=\langle g_i,x_0-x_\star\rangle,\qquad
  G_{ij}=\langle g_i,g_j\rangle,
  \qquad
  v=\sum_{i=0}^N c_i g_i .
\]
Since $x_i=x_0-\alpha\sum_{\ell=0}^{i-1}g_\ell$, every $Q_{ij}$ is a linear
expression in $f_0-f_\star,\ldots,f_N-f_\star$, the $y_i$, and the Gram
entries $G_{ij}$.  With the weights displayed above, the resulting coefficient
identity is
\begin{equation}
\begin{aligned}
  \sum_{i,j\in\{\star,0,\ldots,N\}}\lambda_{ij}Q_{ij}
  &=
  f_\star-f_N
  +\left\langle v,x_0-x_\star\right\rangle
  -\frac{1}{4r}\|v\|^2                                      \\
  &\quad
  +\sum_{i=0}^{N-1}\varepsilon_i(f_i-f_\star)
  +\frac{\varepsilon_N}{2}\|g_0\|^2 .
\end{aligned}
\label{eq:dual-identity}
\end{equation}
The coefficient verification for \eqref{eq:dual-identity} is given in
Appendix~\ref{app:dual-coeff}.  There the function-value coefficients,
the $y_i$ coefficients, the adjacent and reversed adjacent Gram coefficients,
the diagonal Gram coefficients, and the nonadjacent off-diagonal coefficients
are listed separately.  The result is exactly Systems C, A, B, and D.

If $a,b,c,d$ are positive, then every displayed $\lambda_{ij}$ is
nonnegative.  Since $f$ is convex and $1$-smooth, every interpolation slack
$Q_{ij}$ is nonnegative, and therefore the left side of
\eqref{eq:dual-identity} is nonnegative.  If
$\varepsilon_0=\cdots=\varepsilon_N=0$, the residual terms vanish and
\[
  0
  \le
  f_\star-f_N
  +\left\langle v,x_0-x_\star\right\rangle
  -\frac{1}{4r}\|v\|^2 .
\]
For any vectors $v,z$ and any $r>0$,
\[
  \langle v,z\rangle-\frac{1}{4r}\|v\|^2\le r\|z\|^2,
\]
because the difference between the right and left sides is
$\|\frac{1}{2\sqrt r}v-\sqrt r\,z\|^2$.  Taking $z=x_0-x_\star$ yields
\[
  f_N-f_\star\le r\|x_0-x_\star\|^2.
\]
Finally, here $r=\rho_N^{2N}/2$, and the root equation gives
$r=1/(2(2N\alpha+1))$ for $\alpha=1+\rho_N$.
\end{proof}

The rest of the paper constructs positive $a,b,c,d$ satisfying
\eqref{eq:C1}--\eqref{eq:Dterm2}, so Proposition~\ref{prop:cert-rate}
applies without any additional certificate input.

\section{The reduced system}
\label{sec:reduced}
The full residual system has \(N+1\) equations, whereas the reduced
unknown \(d\) has only \(N-1\) coordinates.  The last two residuals are
therefore not treated as independent equations at this stage.  Instead,
we first solve the reduced system consisting of the first \(N-1\)
residuals.  Section~\ref{sec:terminal} proves that, at any nonnegative reduced zero, the
two terminal residuals vanish automatically.  This is why the following
definition is the correct reduced problem rather than a relaxation that
loses information.
Let
\[
E_N(d)=(\varepsilon_0(d),\ldots,\varepsilon_{N-2}(d)).
\]
The reduced problem is to solve $E_N(d)=0$ with $d_i>0$.

\begin{theorem}[Positive reduced zero]
\label{thm:positive-reduced-zero}
For every $N\ge3$, the reduced system $E_N(d)=0$ has a solution
$d^\star\in\mathbb{R}_{>0}^{N-1}$.
\end{theorem}

The proof is given in Section~\ref{sec:existence}.  We use this same vector
$d^\star$ throughout the remainder of the paper.  There is no replacement by
a different branch.

\section{Existence of the positive reduced zero}
\label{sec:existence}

The existence proof uses two ingredients: boundary signs on a simplex and a
Brouwer fixed-point argument.  It produces the interior zero used in all
subsequent sections.

\subsection{Boundary signs}

Let
\[
  \Sigma_N=\{d\in\mathbb{R}_{\ge0}^{N-1}:\sum_{i=0}^{N-2}d_i\le 2N\}.
\]
For the boundary sign proof we use the prefix-tail notation developed
in Appendix~\ref{app:lower-face}.  The prefixes \(P_m\) are nondecreasing whenever
\(d\ge0\), and the tails \(T_m\) decrease by these prefixes.  Thus, on a
coordinate face \(d_i=0\), the relevant tail variables can be bounded by
a single scalar affine majorant.  This converts the residual sign
problem into an elementary one-variable estimate involving the geometric
ratio \(v=2\rho-1\in(0,1)\).

\begin{lemma}[Lower coordinate faces]
\label{lem:lower-faces}
At the GSW root $\rho=\rho_N$, for every $N\ge3$, every
$0\le i\le N-2$, and every $d\ge0$,
\[
  d_i=0\quad\Longrightarrow\quad \varepsilon_i(d)<0.
\]
\end{lemma}

\begin{proof}
We use the tail-prefix notation and the residual normal forms proved in
Appendix~\ref{app:lower-face}.  These identities are algebraic
consequences of Systems C, A, and B and of the definitions of the
residuals; they do not use the reduced equations \(E_N(d)=0\) or the
terminal identity \(L_N(d)=0\). The two scalar estimates used below are recorded at the end of
Appendix~\ref{app:lower-face}; the proof there expands the geometric
moment sums and checks the \(x^2\), \(x\), and constant coefficients
separately.  Put
\[
  u=1-\rho,\qquad v=2\rho-1,\qquad R=\rho^{2N}.
\]
By Lemma~\ref{lem:rho-root-location}, \(1/2<\rho<1\), hence
\(0<u<1/2\) and \(0<v<1\).

We first consider the face \(d_0=0\).  Then \(P_0=P_1=1\).  Let
\[
  x=T_1 .
\]
Since \(d\ge0\), the prefixes \(P_1,\ldots,P_{N-1}\) are nondecreasing.
For \(1\le r\le N-2\),
\[
  T_1-T_{r+1}=P_2+\cdots+P_{r+1}\ge r,
\]
and therefore
\[
  0\le T_{r+1}\le x-r .
\]
By \eqref{eq:eps0-tail-face} and \eqref{eq:lower-B-tail-normal},
\[
  \frac{\varepsilon_0}{R}=uB_0-(1+\rho),
\]
where
\[
  B_0=T_0T_1+(\rho-2)T_1^2
      +2u^2\sum_{\ell=2}^{N-1}v^{\ell-2}T_\ell^2 .
\]
Since \(T_0=P_1+T_1=x+1\), \(T_1=x\), and
\(T_{r+1}^2\le (x-r)^2\) for \(1\le r\le N-2\), we get
\[
  \frac{\varepsilon_0}{R}
  \le
  u\left\{
    x(x+1)+(\rho-2)x^2
    +2u^2\sum_{r\ge1}v^{r-1}(x-r)^2
  \right\}
  -(1+\rho).
\]
The extension to the infinite sum only adds nonnegative terms.  The last
display is exactly the scalar first-face estimate
\eqref{eq:lower-first-face-scalar} from Appendix~\ref{app:lower-face}.  That
estimate depends only on the geometric-arithmetic sums
\(\sum_{r\ge1}v^{r-1}\), \(\sum_{r\ge1}rv^{r-1}\),
\(\sum_{r\ge1}r^2v^{r-1}\), and on \(1-v=2u\); it does not use the reduced
zero equations, terminal completion, or any positivity result from later
sections.  Hence
\[
  \frac{\varepsilon_0}{R}\le -\frac{2+\rho}{2}<0.
\]
Thus \(\varepsilon_0(d)<0\) on the face \(d_0=0\).

Now let \(1\le i\le N-2\) and assume \(d_i=0\).  Set
\[
  p=P_i=P_{i+1}>0,\qquad x=\frac{T_{i+1}}{p}.
\]
Then
\[
  T_{i+1}=px,\qquad T_i=p(x+1),\qquad T_{i-1}=p(x+2).
\]
For \(1\le r\le N-i-2\),
\[
  T_{i+1}-T_{i+1+r}
  =
  P_{i+2}+\cdots+P_{i+1+r}
  \ge rp,
\]
because each displayed prefix is at least \(p\).  Hence
\[
  0\le T_{i+1+r}\le p(x-r),
\]
and so \(T_{i+1+r}^2\le p^2(x-r)^2\).  When \(i=N-2\), this range is
empty, and the corresponding finite tail sum is empty before we extend
it to the infinite nonnegative upper bound below.

By \eqref{eq:eps-tail-lower-face},
\[
  \frac{\varepsilon_i}{R}
  =
  -T_{i-1}^2+\gamma_1T_i^2+\gamma_2T_{i+1}^2
  +4u^4\sum_{\ell=i+2}^{N-1}v^{\ell-i-2}T_\ell^2,
\]
where
\[
  \gamma_1=\rho^2-2\rho+3,\qquad
  \gamma_2=(\rho-2)(2\rho^2-3\rho+2).
\]
Using the preceding bounds and extending the finite positive tail to an
infinite one gives
\[
  \frac{\varepsilon_i}{Rp^2}
  \le
  -(x+2)^2+\gamma_1(x+1)^2+\gamma_2x^2
  +4u^4\sum_{r\ge1}v^{r-1}(x-r)^2 .
\]
The right-hand side is exactly the middle-face scalar estimate
\eqref{eq:lower-middle-face-scalar} from Appendix~\ref{app:lower-face}.  As
above, that estimate is an independent algebraic consequence of the three
geometric-arithmetic sums and \(1-v=2u\), and it does not use the reduced zero
equations or terminal completion.  Hence
\[
  \frac{\varepsilon_i}{Rp^2}\le -(1+\rho)<0.
\]
Therefore \(\varepsilon_i(d)<0\) for every face \(d_i=0\),
\(1\le i\le N-2\).
\end{proof}

The following expansion is obtained by collecting coefficients in the
polynomial \(\sum_{i=0}^{N-2}\varepsilon_i(d)\).  We display the
coefficient classes separately so that the calculation can be checked
without following the full expanded expression term by term.
\begin{lemma}[Outer face]
\label{lem:outer-face}
Let $R_d=\sum_{i=0}^{N-2}d_i$.  On the outer face $R_d=2N$ of $\Sigma_N$,
\[
  \sum_{i=0}^{N-2}\varepsilon_i(d)>0.
\]
\end{lemma}

\begin{proof}
Set $q=\rho^N$.  Summing \eqref{eq:D0}--\eqref{eq:Dmid} first cancels the
interior $a,b$ chain and gives
\[
\begin{aligned}
  \sum_{i=0}^{N-2}\varepsilon_i
  &=
  a_{N-2}-b_{N-2}
  +\sum_{i=0}^{N-2}d_i\sum_{j=i+2}^Nc_j
  -c_0
  -\sum_{i=1}^{N-2}c_i\left(1+\sum_{h=0}^{i-2}d_h\right).
\end{aligned}
\]
From \eqref{eq:Apen}--\eqref{eq:Bpen}, with $\alpha=1+\rho$,
\[
\begin{aligned}
  a_{N-2}-b_{N-2}
  &=
  \frac{
   (1-\rho)c_{N-1}^2/R
   +2c_{N-2}c_{N-1}/R
   -(1-\rho)a_{N-1}}
       {1+\rho}                                      \\
  &\quad
  -\frac{(2+\rho)c_{N-1}\left(1+\sum_{h=0}^{N-3}d_h\right)}
        {1+\rho}.
\end{aligned}
\]
Using \eqref{eq:Aterm} and System C in the last two displays gives the
following degree-two polynomial in the variables $d_i$; the table records the
origin of each coefficient class:
\[
\resizebox{\textwidth}{!}{$
\begin{array}{c|c|c}
\text{piece using }\eqref{eq:Apen}\text{--}\eqref{eq:Aterm}\text{ and System C}
& \text{contribution}
& \text{coefficient class}\\ \hline
\text{terminal square part}
& q^2R_d^2
& [d_jd_k]=2q^2\ (j<k),\ [d_j^2]=q^2\\[1mm]
\text{prefix-length part of System C}
& q^2(1+\rho)\sum_{j=1}^{N-2}j\,d_j
& q^2(1+\rho)j\text{ in }[d_j]\\[1mm]
\text{remaining linear mass}
& R_d\{2q-q^2((N-1)\rho+N-3)\}
& \text{uniform part of }[d_j]\\[1mm]
\text{value at }d=0
& -(N-1)(1+\rho)q^2+q(1-\rho)-(1-\rho)^2
& \text{constant term}.
\end{array}
$}
\]
The first row is the only quadratic source.  The second row is the only
source that depends on the index \(j\) of a linear coefficient; all other
linear terms are collected in \(R_d\).  Hence the coefficient classes are
\[
  [d_jd_k]\sum_{i=0}^{N-2}\varepsilon_i
  =
  \begin{cases}
    2q^2, & j<k,\\
    q^2, & j=k,
  \end{cases}
\]
\[
  [d_j]\sum_{i=0}^{N-2}\varepsilon_i
  =
  2q-q^2((N-1)\rho+N-3)+q^2(1+\rho)j,
  \qquad 0\le j\le N-2,
\]
and
\[
  [1]\sum_{i=0}^{N-2}\varepsilon_i
  =
  -(N-1)(1+\rho)q^2+q(1-\rho)-(1-\rho)^2.
\]
Equivalently,
\begin{align*}
  \sum_{i=0}^{N-2}\varepsilon_i(d)
  &=
  q^2R_d^2
  +R_d\{2q-q^2((N-1)\rho+N-3)\}             \\
  &\quad
  -(N-1)(1+\rho)q^2+q(1-\rho)-(1-\rho)^2
  +q^2(1+\rho)\sum_{j=1}^{N-2}j d_j .
\end{align*}
The last term is nonnegative.  On $R_d=2N$ the remaining part is bounded below
by
\[
  q^2\{2N^2+5N+1-(N-1)(2N+1)\rho\}
  +q(4N+1-\rho)-(1-\rho)^2.
\]
Since $\rho<1$,
\[
  2N^2+5N+1-(N-1)(2N+1)\rho>6N+2.
\]
The root equation gives $q^2>1/(4N+1)$, so the first term is larger than
$(6N+2)/(4N+1)>1$.  The middle term is positive, while
$(1-\rho)^2<1/4$ because $\rho>1/2$.  Hence the sum is positive.
\end{proof}

\begin{lemma}[Simplex existence]
\label{lem:simplex-existence}
Let \(n\ge1\), let \(\mathcal R>0\), let
\[
  K=\left\{x\in\mathbb R_{\ge0}^n:\sum_{i=0}^{n-1}x_i\le \mathcal R\right\},
\]
and let \(F:K\to\mathbb R^n\) be continuous. Suppose that, for each
\(0\le i\le n-1\),
\[
  F_i(x)<0
  \qquad\text{whenever }x\in K\text{ and }x_i=0,
\]
and that
\[
  \sum_{i=0}^{n-1}F_i(x)>0
  \qquad\text{whenever }x\in K\text{ and }\sum_{i=0}^{n-1}x_i=\mathcal R.
\]
Then \(F\) has a zero in the interior of \(K\), that is, there exists
\(x\in K\) such that
\[
  F(x)=0,\qquad x_i>0\ (0\le i\le n-1),\qquad
  \sum_{i=0}^{n-1}x_i<\mathcal R .
\]
\end{lemma}

\begin{proof}
For each coordinate face
\[
  K_i=\{x\in K:x_i=0\},
\]
compactness and the strict inequality \(F_i<0\) on \(K_i\) give
\[
  m_i:=\max_{x\in K_i}F_i(x)<0.
\]
Set
\[
  \delta_i=-\frac{m_i}{2}>0.
\]
Since \(F_i\) is uniformly continuous on the compact set \(K\), there is
\(\eta_i>0\) such that
\[
  |F_i(x)-F_i(y)|\le \delta_i
  \qquad\text{whenever }x,y\in K\text{ and }\|x-y\|\le \eta_i .
\]
We claim that
\[
  0\le x_i\le \eta_i,\quad x\in K
  \quad\Longrightarrow\quad
  F_i(x)\le -\delta_i .
\]
Indeed, if \(0\le x_i\le\eta_i\), define
\[
  y=x-x_i e_i,
\]
where \(e_i\) is the \(i\)-th coordinate vector.  Then \(y\in K_i\) and
\(\|x-y\|=x_i\le\eta_i\).  Therefore
\[
  F_i(x)\le F_i(y)+\delta_i\le m_i+\delta_i=-\delta_i .
\]

Now let
\[
  K_{\rm out}=\left\{x\in K:\sum_{i=0}^{n-1}x_i=\mathcal R\right\},
\]
and put
\[
  G(x)=\sum_{i=0}^{n-1}F_i(x).
\]
By compactness and the strict outer-face inequality,
\[
  m_{\rm out}:=\min_{x\in K_{\rm out}}G(x)>0.
\]
Set
\[
  \delta_{\rm out}=\frac{m_{\rm out}}{2}>0.
\]
Since \(G\) is uniformly continuous on \(K\), there is
\(\eta_{\rm out}>0\) such that
\[
  |G(x)-G(y)|\le\delta_{\rm out}
  \qquad\text{whenever }x,y\in K\text{ and }\|x-y\|\le\eta_{\rm out}.
\]
We claim that
\[
  \mathcal R-\eta_{\rm out}\le \sum_{i=0}^{n-1}x_i\le\mathcal R,
  \quad x\in K
  \quad\Longrightarrow\quad
  G(x)\ge\delta_{\rm out}.
\]
Indeed, write
\[
  s=\sum_{i=0}^{n-1}x_i.
\]
If \(s\ge\mathcal R-\eta_{\rm out}\), define
\[
  y=x+(\mathcal R-s)e_0.
\]
Then \(y\in K_{\rm out}\), \(y\ge0\), and
\[
  \|x-y\|=\mathcal R-s\le\eta_{\rm out}.
\]
Hence
\[
  G(x)\ge G(y)-\delta_{\rm out}\ge m_{\rm out}-\delta_{\rm out}
  =\delta_{\rm out}.
\]

Next define
\[
  M
  =
  1+\max_{x\in K}
  \max\left\{
    |F_0(x)|,\ldots,|F_{n-1}(x)|,\,
    \left|\sum_{i=0}^{n-1}F_i(x)\right|
  \right\}.
\]
Then \(M>0\).  Choose \(\lambda>0\) so small that
\[
  \lambda M\le \frac12\min_{0\le i\le n-1}\eta_i,
  \qquad
  \lambda M\le \frac12\eta_{\rm out}.
\]
Define
\[
  \mathcal T:K\to\mathbb R^n,\qquad \mathcal T(x)=x-\lambda F(x).
\]
We prove that \(\mathcal T(K)\subset K\).

First fix a coordinate \(i\).  If \(x_i\le\eta_i\), then by the strip
estimate near the coordinate face,
\[
  \mathcal T_i(x)=x_i-\lambda F_i(x)
  \ge x_i+\lambda\delta_i\ge0.
\]
If \(x_i>\eta_i\), then
\[
  \mathcal T_i(x)
  =
  x_i-\lambda F_i(x)
  \ge x_i-\lambda |F_i(x)|
  \ge \eta_i-\lambda M
  \ge \frac{\eta_i}{2}>0.
\]
Thus every coordinate of \(\mathcal T(x)\) is nonnegative.

It remains to check the outer constraint.  Let
\[
  s=\sum_{i=0}^{n-1}x_i.
\]
If \(s\ge\mathcal R-\eta_{\rm out}\), then the outer strip estimate gives
\(G(x)\ge\delta_{\rm out}\), and therefore
\[
  \sum_{i=0}^{n-1}\mathcal T_i(x)
  =
  s-\lambda G(x)
  \le s
  \le \mathcal R.
\]
If \(s<\mathcal R-\eta_{\rm out}\), then
\[
  \sum_{i=0}^{n-1}\mathcal T_i(x)
  =
  s-\lambda G(x)
  \le s+\lambda |G(x)|
  \le \mathcal R-\eta_{\rm out}+\lambda M
  \le \mathcal R-\frac{\eta_{\rm out}}{2}
  <\mathcal R.
\]
Hence \(\mathcal T(x)\in K\) for every \(x\in K\), so
\(\mathcal T(K)\subset K\).

The set \(K\) is nonempty, compact, and convex, and \(\mathcal T\) is
continuous.  By Brouwer's fixed-point theorem, see, e.g.,
Granas and Dugundji~\cite{granas-dugundji}, there exists \(x\in K\) such
that
\[
  \mathcal T(x)=x.
\]
Since \(\lambda>0\), this implies
\[
  F(x)=0.
\]

Finally, this zero cannot lie on the boundary of \(K\).  If \(x_i=0\) for
some \(i\), then the coordinate-face assumption gives \(F_i(x)<0\),
contradicting \(F(x)=0\).  If
\[
  \sum_{i=0}^{n-1}x_i=\mathcal R,
\]
then the outer-face assumption gives
\[
  \sum_{i=0}^{n-1}F_i(x)>0,
\]
again contradicting \(F(x)=0\).  Therefore
\[
  x_i>0\quad(0\le i\le n-1),
  \qquad
  \sum_{i=0}^{n-1}x_i<\mathcal R,
\]
so the zero lies in the interior of \(K\).
\end{proof}

\begin{proof}[Proof of Theorem~\ref{thm:positive-reduced-zero}]
Apply Lemma~\ref{lem:simplex-existence} to $K=\Sigma_N$ and $F=E_N$.
Systems C, A, and B define $c,a,b$ from $d$ by finite algebraic operations
with denominator $\alpha=1+\rho>0$; hence each $\varepsilon_i(d)$ is a
polynomial in $d$, and $E_N$ is continuous on $\Sigma_N$.
Lemma~\ref{lem:lower-faces} supplies the coordinate-face signs, and
Lemma~\ref{lem:outer-face} supplies the outer-face sign.  Therefore
$E_N$ has an interior zero $d^\star$, and all coordinates of $d^\star$ are
positive.
\end{proof}

\section{Terminal residual completion}
\label{sec:terminal}

The reduced system contains only
\(\varepsilon_0,\ldots,\varepsilon_{N-2}\).  This section proves that the
two omitted terminal residuals \(\varepsilon_{N-1}\) and \(\varepsilon_N\)
vanish automatically at every nonnegative reduced zero.  The argument has
three parts.  First, two conservation identities reduce the terminal
residuals to one scalar quantity \(L_N(d)\).  Second, a weighted quadratic
identity, proved coefficient by coefficient in Appendix~\ref{app:terminal},
forces \(L_N(d)\) to satisfy a scalar quadratic equation at a reduced zero.
Third, the GSW root equation and a positivity estimate exclude the spurious
factor of that scalar quadratic.

Recall that \(q=\rho^N\) and \(R=q^2=\rho^{2N}\).  Let
\[
  W_N(d)=\sum_{j=0}^{N-2}(N-1-j)d_j
\]
and define
\begin{equation}
  L_N(d)=q^2(W_N(d)+N-1)+\frac{q^2-1}{1+\rho}+q .
  \label{eq:terminal-L-def}
\end{equation}

\begin{lemma}[Prefix form of System C]
\label{lem:terminal-prefix-c}
Define
\[
  P_0=1,\qquad
  P_m=1+\sum_{h=0}^{m-1}d_h\quad(1\le m\le N-1),
  \qquad P_N=q^{-1}.
\]
Then System C is equivalently
\[
  c_i=R(P_i+\rho P_{i+1})\quad(0\le i\le N-1),
  \qquad c_N=RP_N=q .
\]
\end{lemma}

\begin{proof}
For \(0\le i\le N-2\),
\[
\begin{aligned}
  P_i+\rho P_{i+1}
  &=1+\sum_{h=0}^{i-1}d_h
    +\rho\left(1+\sum_{h=0}^{i}d_h\right)  \\
  &=(1+\rho)\left(1+\sum_{h=0}^{i}d_h\right)-d_i
   =\alpha\sum_{h=0}^{i}d_h-d_i+\alpha ,
\end{aligned}
\]
which is \eqref{eq:C1}.  For \(i=N-1\),
\[
  P_{N-1}+\rho P_N
  =
  1+\sum_{h=0}^{N-2}d_h+\frac{\rho}{q},
\]
which is \eqref{eq:C2}.  Finally
\[
  RP_N=Rq^{-1}=q,
\]
which is \eqref{eq:C3}.
\end{proof}

\begin{lemma}[Right conservation identity]
\label{lem:right-terminal-identity}
For every \(d\),
\[
  \sum_{i=0}^{N-1}\varepsilon_i=-(1+\rho)L_N(d).
\]
Consequently
\[
  \varepsilon_{N-1}+(1+\rho)L_N(d)
  =
  -\sum_{i=0}^{N-2}\varepsilon_i .
\]
\end{lemma}

\begin{proof}
Sum \eqref{eq:D0}, \eqref{eq:Dmid}, and \eqref{eq:Dterm1}.  The
\(a,b\)-part telescopes:
\[
\begin{aligned}
 &(a_0-b_0)
 +\sum_{i=1}^{N-2}(b_{i-1}+a_i-a_{i-1}-b_i)
 +(b_{N-2}+a_{N-1}-a_{N-2})      \\
 &\qquad =a_{N-1}.
\end{aligned}
\]
The \(d\)-weighted double sum is reindexed as
\[
  \sum_{i=0}^{N-2}d_i\sum_{j=i+2}^Nc_j
  =
  \sum_{j=2}^Nc_j\sum_{i=0}^{j-2}d_i
  =
  \sum_{j=2}^Nc_j(P_{j-1}-1).
\]
Therefore
\[
\begin{aligned}
  \sum_{i=0}^{N-1}\varepsilon_i
  &=
  a_{N-1}
  +\sum_{j=2}^Nc_j(P_{j-1}-1)
  -c_0-\sum_{i=1}^{N-1}c_iP_{i-1}.
\end{aligned}
\]
Using the terminal equation of System A,
\[
  a_{N-1}=1-c_NP_{N-1},
\]
the \(c_jP_{j-1}\)-terms cancel:
\[
\begin{aligned}
  &1-c_NP_{N-1}
  +\sum_{j=2}^Nc_jP_{j-1}
  -\sum_{j=2}^Nc_j
  -c_0-\sum_{j=1}^{N-1}c_jP_{j-1}        \\
  &\qquad =
  1-\sum_{j=0}^Nc_j .
\end{aligned}
\]
By Lemma~\ref{lem:terminal-prefix-c},
\[
\begin{aligned}
  \sum_{j=0}^Nc_j
  &=R\sum_{j=0}^{N-1}(P_j+\rho P_{j+1})+RP_N       \\
  &=R\left\{
      P_0+(1+\rho)\sum_{j=1}^{N-1}P_j+(1+\rho)P_N
    \right\}.
\end{aligned}
\]
Since
\[
  \sum_{j=1}^{N-1}P_j=N-1+\sum_{h=0}^{N-2}(N-1-h)d_h
  =N-1+W_N(d),
  \qquad P_N=q^{-1},
\]
we obtain
\[
\begin{aligned}
  1-\sum_{j=0}^Nc_j
  &=
  1-R-R(1+\rho)(W_N(d)+N-1)-(1+\rho)q       \\
  &=
  -(1+\rho)\left\{
      R(W_N(d)+N-1)+\frac{R-1}{1+\rho}+q
    \right\}                                \\
  &=-(1+\rho)L_N(d).
\end{aligned}
\]
The displayed consequence follows by subtracting
\(\sum_{i=0}^{N-2}\varepsilon_i\) from both sides.
\end{proof}

\begin{lemma}[Weighted left conservation identity]
\label{lem:weighted-left-conservation}
For every \(d\),
\[
  \varepsilon_N+\sum_{i=0}^{N-1}\rho^{2i}\varepsilon_i=0 .
\]
\end{lemma}

\begin{proof}
Let \(\mathcal A_i\) and \(\mathcal B_i\) denote the cleared left sides of
Systems A and B.  More precisely, set
\[
  \mathcal A_{N-1}:=a_{N-1}+c_NP_{N-1}-1 ,
\]
and for \(0\le i\le N-3\),
\[
\begin{aligned}
  \mathcal A_i
  &:=(1+\rho)a_i-\frac{c_{i+1}^2}{R}-\frac{c_ic_{i+1}}{R}+a_{i+1}
     +(2+\rho)c_{i+1}P_i
     +d_{i+1}\sum_{j=i+3}^Nc_j-(1+2\rho)b_{i+1},
\end{aligned}
\]
while the penultimate row is
\[
  \mathcal A_{N-2}:=(1+\rho)a_{N-2}
  -\frac{c_{N-1}^2}{R}-\frac{c_{N-2}c_{N-1}}{R}+a_{N-1}
  +(2+\rho)c_{N-1}P_{N-2}.
\]
Similarly, for \(0\le i\le N-3\),
\[
\begin{aligned}
  \mathcal B_i
  &:=(1+\rho)b_i-\frac{\rho c_{i+1}^2}{R}
     +\frac{c_ic_{i+1}}{R}
     +\rho a_{i+1}-c_{i+1}P_i
     +\rho d_{i+1}\sum_{j=i+3}^Nc_j
     -\rho(1+2\rho)b_{i+1},
\end{aligned}
\]
and
\[
  \mathcal B_{N-2}:=(1+\rho)b_{N-2}
  -\frac{\rho c_{N-1}^2}{R}
  +\frac{c_{N-2}c_{N-1}}{R}
  +\rho a_{N-1}-c_{N-1}P_{N-2}.
\]
Systems A and B state that all \(\mathcal A_i\) and \(\mathcal B_i\)
vanish.  The terminal row \(\mathcal A_{N-1}\) is part of System A and is
used in Lemma~\ref{lem:right-terminal-identity}; it is not needed with a
nonzero multiplier in the weighted identity below.

Put
\[
  W:=\varepsilon_N+\sum_{i=0}^{N-1}\rho^{2i}\varepsilon_i .
\]
We first decompose \(W\) into cleared A/B rows plus a pure \(c,P\)-bracket:
\begin{equation}
\begin{aligned}
 W
 &=
 \sum_{i=0}^{N-2}
   \left(-\frac{\rho^{2i+2}}{1+\rho}\right)\mathcal A_i
 +\sum_{i=0}^{N-2}
   \left(\frac{\rho^{2i+1}(2+\rho)}{1+\rho}\right)\mathcal B_i       \\
 &\quad
 +\frac1R\left[
  c_0(c_0-2R)
  +\sum_{s=1}^{N-1}\rho^{2s-1}c_s
    \bigl(2RP_{s-1}-2c_{s-1}+\rho c_s\bigr)
 \right].
\end{aligned}
\label{eq:weighted-left-decomposition}
\end{equation}
For completeness, we record the coefficient check for the variables that
are not in the final bracket.

For the \(a\)-coefficients,
\[
\begin{array}{c|c|c}
\text{index} & [a_k]W & [a_k]\text{ right side of }\eqref{eq:weighted-left-decomposition}\\ \hline
k=0
& -\rho^2
& -\rho^2\\[1mm]
1\le k\le N-2
& \rho^{2k}-\rho^{2k+2}
& -\rho^{2k+2}-\dfrac{\rho^{2k}}{1+\rho}
  +\dfrac{\rho^{2k}(2+\rho)}{1+\rho}\\[3mm]
k=N-1
& \rho^{2N-2}
& -\dfrac{\rho^{2N-2}}{1+\rho}
  +\dfrac{\rho^{2N-2}(2+\rho)}{1+\rho}.
\end{array}
\]
The second and third right-hand entries reduce respectively to
\(\rho^{2k}-\rho^{2k+2}\) and \(\rho^{2N-2}\).

For the \(b\)-coefficients,
\[
\begin{array}{c|c|c}
\text{index} & [b_k]W & [b_k]\text{ right side of }\eqref{eq:weighted-left-decomposition}\\ \hline
k=0
& (1+2\rho)-1+\rho^2=\rho(2+\rho)
& \rho(2+\rho)\\[1mm]
1\le k\le N-2
& \rho^{2k+2}-\rho^{2k}
& \rho^{2k+1}(2+\rho)
  +\dfrac{\rho^{2k}(1+2\rho)}{1+\rho}
  -\dfrac{\rho^{2k}(2+\rho)(1+2\rho)}{1+\rho}.
\end{array}
\]
The second right-hand entry simplifies to
\(\rho^{2k+2}-\rho^{2k}\).  There is no \(b_{N-1}\)-row.

It remains to check the \(c\)-terms.  The \(d\)-tails in the weighted
residuals are collected by
\[
  \sum_{i=0}^{s-2}d_i=P_{s-1}-1,
\]
where the left-hand side is empty when \(s=1\).  After this reindexing,
the non-\(a,b\) part of \(W\), before pulling out \(1/R\), is
\[
  c_0^2-2Rc_0
  +\sum_{s=1}^{N-1}\rho^{2s-1}
    \left(2Rc_sP_{s-1}-2c_{s-1}c_s+\rho c_s^2\right).
\]
Equivalently, for each fixed \(1\le s\le N-1\), the relevant local
coefficient check is
\begingroup
\small
\setlength{\arraycolsep}{4pt}
\renewcommand{\arraystretch}{1.25}

\[
\begin{array}{c|ccccc}
\text{family}
& W
& T_A
& T_B
& T_R
& \text{total}\\ \hline
c_sP_{s-1}
& -\rho^{2s}
& -\dfrac{\rho^{2s}(2+\rho)}{1+\rho}
& -\dfrac{\rho^{2s-1}(2+\rho)}{1+\rho}
& 2\rho^{2s-1}
& -\rho^{2s}\\
c_{s-1}c_s
& 0
& \dfrac{\rho^{2s}}{(1+\rho)R}
& \dfrac{\rho^{2s-1}(2+\rho)}{(1+\rho)R}
& -\dfrac{2\rho^{2s-1}}R
& 0\\
c_s^2
& 0
& \dfrac{\rho^{2s}}{(1+\rho)R}
& -\dfrac{\rho^{2s}(2+\rho)}{(1+\rho)R}
& \dfrac{\rho^{2s}}R
& 0 .
\end{array}
\]

\[
T_A=-\frac{\rho^{2s}}{1+\rho}\mathcal A_{s-1},
\qquad
T_B=\frac{\rho^{2s-1}(2+\rho)}{1+\rho}\mathcal B_{s-1},
\qquad
T_R=\frac1R\text{ bracket}.
\]

\endgroup
For \(s=N-1\), the rows \(\mathcal A_{s-1}\) and \(\mathcal B_{s-1}\)
mean the penultimate rows \(\mathcal A_{N-2}\) and \(\mathcal B_{N-2}\).
The first total follows from
\[
  \rho^{2s-1}\left(
    -\frac{\rho(2+\rho)}{1+\rho}
    -\frac{2+\rho}{1+\rho}
    +2
  \right)
  =-\rho^{2s},
\]
and the second and third totals vanish by
\[
  \frac{\rho+2+\rho}{1+\rho}-2=0,
  \qquad
  \frac{1-(2+\rho)}{1+\rho}+1=0.
\]
This proves \eqref{eq:weighted-left-decomposition}.

Since Systems A and B give \(\mathcal A_i=\mathcal B_i=0\), it remains to
prove that the bracket in \eqref{eq:weighted-left-decomposition} is zero.
By Lemma~\ref{lem:terminal-prefix-c},
\[
  c_0(c_0-2R)
  =R^2(P_0+\rho P_1)(\rho P_1-P_0).
\]
For \(1\le s\le N-1\),
\[
\begin{aligned}
  2RP_{s-1}-2c_{s-1}+\rho c_s
  &=2RP_{s-1}-2R(P_{s-1}+\rho P_s)
    +\rho R(P_s+\rho P_{s+1})       \\
  &=R\rho(\rho P_{s+1}-P_s).
\end{aligned}
\]
Therefore the bracket equals
\[
  R^2\left[
  (P_0+\rho P_1)(\rho P_1-P_0)
  +\sum_{s=1}^{N-1}\rho^{2s}
    (P_s+\rho P_{s+1})(\rho P_{s+1}-P_s)
  \right].
\]
Each summand is a difference of squares:
\[
  \rho^{2s}(P_s+\rho P_{s+1})(\rho P_{s+1}-P_s)
  =
  \rho^{2s+2}P_{s+1}^2-\rho^{2s}P_s^2 .
\]
Including the initial term, the sum telescopes to
\[
  R^2(-P_0^2+\rho^{2N}P_N^2)=0,
\]
because \(P_0=1\) and \(P_N=q^{-1}=\rho^{-N}\).  Hence \(W=0\).
\end{proof}

\begin{lemma}[Left terminal residual identity]
\label{lem:left-terminal-identity}
For every \(d\),
\[
  \varepsilon_N-\rho^{2N-2}(1+\rho)L_N(d)
  =
  \sum_{i=0}^{N-2}\bigl(\rho^{2N-2}-\rho^{2i}\bigr)\varepsilon_i(d).
\]
\end{lemma}

\begin{proof}
Lemma~\ref{lem:weighted-left-conservation} gives
\[
  \varepsilon_N=-\sum_{i=0}^{N-1}\rho^{2i}\varepsilon_i .
\]
Lemma~\ref{lem:right-terminal-identity} gives
\[
  \varepsilon_{N-1}=-(1+\rho)L_N(d)-\sum_{i=0}^{N-2}\varepsilon_i .
\]
Substituting the second identity into the first gives
\[
\begin{aligned}
  \varepsilon_N
  &=
  -\sum_{i=0}^{N-2}\rho^{2i}\varepsilon_i
  -\rho^{2N-2}\varepsilon_{N-1}                         \\
  &=
  -\sum_{i=0}^{N-2}\rho^{2i}\varepsilon_i
  +\rho^{2N-2}(1+\rho)L_N(d)
  +\rho^{2N-2}\sum_{i=0}^{N-2}\varepsilon_i .
\end{aligned}
\]
Moving the \(L_N\)-term to the left yields the claimed identity.
\end{proof}

\begin{lemma}[Terminal residual reduction]
\label{lem:terminal-reduction}
Modulo the reduced residual equations
\(\varepsilon_0=\cdots=\varepsilon_{N-2}=0\),
\[
  \varepsilon_{N-1}=-(1+\rho)L_N(d),
  \qquad
  \varepsilon_N=\rho^{2N-2}(1+\rho)L_N(d).
\]
\end{lemma}

\begin{proof}
The first identity follows from Lemma~\ref{lem:right-terminal-identity};
the second follows from Lemma~\ref{lem:left-terminal-identity}.
\end{proof}

Define
\[
  \Phi_N(\rho)=\sum_{m=0}^{2N-1}(-1)^m(m+1)\rho^m,
\]
\[
  \Psi_N(\rho)
  =
  2\sum_{m=0}^{2N-1}(-\rho)^m
  -(2N-3)\rho^{2N}
  -\rho^{2N+1}\sum_{m=0}^{2N-4}(-\rho)^m,
\]
and, for \(0\le i\le N-2\),
\[
  \Omega_{N,i}(\rho)=
  \rho^{2N}\sum_{m=0}^{2N-4}
   (-1)^m\min\{2N-3-m,\;2(N-1-i)\}\rho^m .
\]

\begin{lemma}[Terminal quadratic identity]
\label{lem:terminal-quadratic}
For all \(d\),
\[
  L_N(d)^2+\Psi_N(\rho)L_N(d)+\Phi_N(\rho)
  =
  \sum_{i=0}^{N-2}\Omega_{N,i}(\rho)\varepsilon_i(d).
\]
\end{lemma}

\begin{proof}
Appendix~\ref{app:terminal} proves this identity by coefficient comparison.
The dependence is as follows.  First,
\eqref{eq:terminal-eps0-tail} and \eqref{eq:terminal-epsi-tail} rewrite the
reduced residuals in the tail variables.  These two tail formulas are obtained
from Systems C, A, B, and D and from the tail-prefix normal form of
Appendix~\ref{app:lower-face}; they do not use the reduced equations
\(E_N(d)=0\), the terminal completion \(L_N(d)=0\), the tail-square
recurrence, or any positivity result.

Second, Lemma~\ref{lem:capped-first-difference} computes the first
difference of the capped alternating weights \(S_h\).  Third,
Lemma~\ref{lem:interior-tail-square-cancellation} shows that in the weighted
sum
\[
  \sum_{i=0}^{N-2}S_{N-1-i}\frac{\varepsilon_i}{R}
\]
all variable tail-square coefficients except the \(T_0^2\)-coefficient
cancel.  This leaves the scalar tail polynomial displayed in
\eqref{eq:terminal-weighted-tail-polynomial}.  Finally,
Lemma~\ref{lem:endpoint-constant-comparison} matches the remaining \(T_0\)
and constant coefficients, and
Lemma~\ref{lem:terminal-coefficient-classes} converts the tail-variable
identity back into the coefficient identities in
\(d_0,\ldots,d_{N-2}\).

The scaling is also fixed in Appendix~\ref{app:terminal}.  Since
\[
  \Omega_{N,i}=R S_{N-1-i},
\]
while \eqref{eq:terminal-eps0-tail}--\eqref{eq:terminal-epsi-tail} write each
\(\varepsilon_i\) as \(R\) times a tail polynomial, the weighted residual sum
has the common factor \(R^2\).  On the other side,
\[
  L_N=RT_0+\frac{R-1}{1+\rho}.
\]
Thus both sides have the same \(R^2\)-scaled \(T_0^2\)-coefficient, and the
remaining \(T_0\) and constant coefficients are precisely the endpoint
comparisons proved in Appendix~\ref{app:terminal}.
\end{proof}

\begin{lemma}[Exclusion of the terminal spurious factor]
\label{lem:terminal-factor-positive}
At the GSW root, for every \(d\ge0\),
\[
  L_N(d)+\Psi_N(\rho)>0 .
\]
\end{lemma}

\begin{proof}
Since every coefficient in \(W_N(d)\) is nonnegative, \(d\ge0\) gives
\[
  L_N(d)\ge q^2(N-1)+\frac{q^2-1}{1+\rho}+q .
\]
Using the definition of \(\Psi_N\) and summing the finite geometric series,
\[
  L_N(d)+\Psi_N(\rho)\ge
  \frac{1+(1+\rho)q-(N-1)(1+\rho)q^2-\rho^{4N-2}}{1+\rho}.
\]
At the GSW root,
\[
  q^2=\rho^{2N}=\frac{1}{2N(1+\rho)+1}.
\]
Therefore
\[
  (N-1)(1+\rho)q^2
  =
  \frac{(N-1)(1+\rho)}{2N(1+\rho)+1}
  <\frac12 .
\]
Also \(0<\rho<1\) gives
\[
  \rho^{4N-2}<\rho^{2N}=q^2,
\]
and \(1/2<\rho<1\) gives
\[
  q^2=\frac{1}{2N(1+\rho)+1}
  <\frac{1}{3N+1}\le\frac1{10}
  \qquad(N\ge3).
\]
Hence the numerator in the lower bound is strictly larger than
\[
  1-\frac12-\frac1{10}>0,
\]
and the denominator \(1+\rho\) is positive.
\end{proof}

\begin{lemma}[Terminal completion]
\label{lem:terminal-completion}
If \(d\ge0\) and \(E_N(d)=0\), then
\[
  \varepsilon_{N-1}(d)=\varepsilon_N(d)=0 .
\]
\end{lemma}

\begin{proof}
The alternating arithmetic-geometric sum has the closed form
\[
  \Phi_N(\rho)=
  \frac{1-(2N+1)\rho^{2N}-2N\rho^{2N+1}}{(1+\rho)^2}.
\]
The numerator is exactly the GSW root equation, so \(\Phi_N(\rho)=0\) at the
GSW root.  Since \(E_N(d)=0\), the terminal quadratic identity gives
\[
  L_N(d)^2+\Psi_N(\rho)L_N(d)=0,
\]
or equivalently
\[
  L_N(d)\bigl(L_N(d)+\Psi_N(\rho)\bigr)=0.
\]
Lemma~\ref{lem:terminal-factor-positive} excludes the second factor for
\(d\ge0\).  Hence \(L_N(d)=0\).  Lemma~\ref{lem:terminal-reduction} then
implies
\[
  \varepsilon_{N-1}(d)=\varepsilon_N(d)=0 .
\]
\end{proof}
\paragraph{Proof-dependency map.}
The terminal quadratic identity in Lemma~\ref{lem:terminal-quadratic} is
proved in Appendix~\ref{app:terminal} from Systems C, A, and B, tail algebra,
and finite alternating-sum identities.  It does not use the later
tail-square recurrence or the positivity argument.  Lemma~\ref{lem:terminal-completion}
first uses the two conservation identities, the terminal quadratic identity,
and the spurious-factor exclusion to prove
\(\varepsilon_{N-1}=\varepsilon_N=0\).  Only after this terminal completion
does Section~\ref{sec:coeff-pos} use cumulative margins, the A/B bridge, and
the positivity argument for the constructed coefficient vectors.

\section{Positivity of the coefficient vectors}
\label{sec:coeff-pos}

\subsection{Positivity of \texorpdfstring{$c$}{c}}

\begin{lemma}
\label{lem:c-positive}
If $d_i\ge0$ and $\rho>0$, then every $c_i$ defined by System C is strictly
positive.
\end{lemma}

\begin{proof}
For $0\le i\le N-2$,
\[
  c_i
  =
  R\left(\alpha\sum_{k=0}^{i-1}d_k+\rho d_i+\alpha\right),
\]
whose factors are positive.  Also $c_{N-1}=R(1+\sum_kd_k+\rho/q)>0$ and
$c_N=q>0$.
\end{proof}

\subsection{Cumulative margins}

The positivity argument below is cumulative: it does not attempt to prove that
each backward increment defining $b$ is positive.  Instead it proves positivity
of the margins $K_i$, which are tied directly to $b_i$ by the cumulative
residual identity.

For $0\le i\le N-2$ define
\[
  K_i(d)=
  \sum_{h=0}^i c_h
  -
  \left(\sum_{h=0}^i d_h\right)\left(\sum_{j=i+1}^N c_j\right).
\]

\begin{lemma}[Cumulative residual identity]
\label{lem:cumulative}
For $0\le i\le N-2$,
\[
  \sum_{h=0}^i\varepsilon_h
  =
  \frac{1-\rho}{\rho}b_i-K_i.
\]
Consequently, at a reduced zero,
\[
  b_i=\frac{\rho}{1-\rho}K_i.
\]
\end{lemma}

\begin{proof}
Set
\[
  b_{-1}=0,\qquad K_{-1}=0,
\]
and recall that
\[
  P_h=1+\sum_{k=0}^{h-1}d_k\qquad(h\ge1),\qquad P_0=1.
\]
For \(0\le h\le N-2\), the definition of \(K_h\) gives
\[
\begin{aligned}
  K_h-K_{h-1}
  &=
  \left(\sum_{\ell=0}^{h}c_\ell
  -\left(\sum_{\ell=0}^{h}d_\ell\right)
    \left(\sum_{j=h+1}^Nc_j\right)\right)        \\
  &\quad -
  \left(\sum_{\ell=0}^{h-1}c_\ell
  -\left(\sum_{\ell=0}^{h-1}d_\ell\right)
    \left(\sum_{j=h}^Nc_j\right)\right)          \\
  &=
  c_h
  +\left(\sum_{\ell=0}^{h-1}d_\ell\right)c_h
  -d_h\sum_{j=h+1}^Nc_j                         \\
  &=
  P_hc_h-d_h\sum_{j=h+1}^Nc_j .
\end{aligned}
\]
For \(h=0\), this same formula reads
\[
  K_0-K_{-1}=c_0-d_0\sum_{j=1}^Nc_j,
\]
which is consistent with \(P_0=1\).

Next use the A/B bridge
\[
  \rho a_h-b_h=\rho d_hc_{h+1},
  \qquad 0\le h\le N-2,
\]
or equivalently
\[
  a_h=\frac{b_h}{\rho}+d_hc_{h+1}.
\]
For \(h=0\), substituting this into \eqref{eq:D0} gives
\[
\begin{aligned}
  \varepsilon_0
  &=
  \frac{b_0}{\rho}+d_0c_1
  +d_0\sum_{j=2}^Nc_j
  -b_0-c_0                                      \\
  &=
  \frac{1-\rho}{\rho}b_0
  -\left(c_0-d_0\sum_{j=1}^Nc_j\right)          \\
  &=
  \frac{1-\rho}{\rho}(b_0-b_{-1})-(K_0-K_{-1}).
\end{aligned}
\]
For \(1\le h\le N-2\), substitute the bridge for both \(a_h\) and
\(a_{h-1}\) in \eqref{eq:Dmid}:
\[
\begin{aligned}
  \varepsilon_h
  &=
  b_{h-1}
  +\left(\frac{b_h}{\rho}+d_hc_{h+1}\right)
  +d_h\sum_{j=h+2}^Nc_j
  -\left(\frac{b_{h-1}}{\rho}+d_{h-1}c_h\right)
  -b_h
  -c_h\left(1+\sum_{k=0}^{h-2}d_k\right)        \\
  &=
  \frac{1-\rho}{\rho}(b_h-b_{h-1})
  +d_h\sum_{j=h+1}^Nc_j
  -P_hc_h                                      \\
  &=
  \frac{1-\rho}{\rho}(b_h-b_{h-1})-(K_h-K_{h-1}).
\end{aligned}
\]
Thus, for every \(0\le h\le N-2\),
\[
  \varepsilon_h
  =
  \frac{1-\rho}{\rho}(b_h-b_{h-1})-(K_h-K_{h-1}).
\]
Summing this identity from \(h=0\) to \(i\) telescopes both differences and
gives
\[
  \sum_{h=0}^i\varepsilon_h
  =
  \frac{1-\rho}{\rho}b_i-K_i .
\]
At a reduced zero, the left-hand side is zero, and therefore
\[
  b_i=\frac{\rho}{1-\rho}K_i .
\]
\end{proof}

\subsection{Tail-square recurrence}

Let
\[
  P_0=1,\qquad
  P_m=1+\sum_{k=0}^{m-1}d_k\quad(1\le m\le N-1),\qquad
  P_N=q^{-1}.
\]
Here $P_N=q^{-1}$ is a terminal boundary convention, not a prefix obtained
from an additional variable $d_{N-1}$.
Let
\[
  \Pi_m=\sum_{h=1}^mP_h\qquad(0\le m\le N-1),
  \qquad T_m=2N-\Pi_m .
\]
At a nonnegative reduced zero, Lemma~\ref{lem:terminal-completion} gives
$L_N(d)=0$.  Since
\[
  W_N(d)+N-1=\sum_{m=1}^{N-1}P_m,
\]
the equation $L_N=0$ and the root relation imply
\[
  \sum_{m=1}^{N-1}P_m=2N-q^{-1}.
\]
Therefore
\[
  T_0=2N,\qquad T_{N-1}=q^{-1},\qquad T_N:=0,\qquad
  P_m=T_{m-1}-T_m\quad(1\le m\le N).
\]
The endpoint convention \(T_N=0\) is used only to make terminal tail-sum
formulas valid at \(k=N\); it introduces no additional variable.

Define
\[
  A_L(\rho)=\sum_{s=0}^L(-1)^s(L+1-s)\rho^s.
\]
Set
\[
  X_n=\rho^{2N}T_{N-1-n}^2,\qquad 0\le n\le N-1.
\]
Thus $X_0=1$.

Let
\[
  \gamma_0=-1,\qquad
  \gamma_1=\rho^2-2\rho+3,\qquad
  \gamma_2=(\rho-2)(2\rho^2-3\rho+2),
\]
and for $s\ge3$,
\[
  \gamma_s=4(\rho-1)^4(2\rho-1)^{s-3}.
\]

\begin{lemma}[Tail-square residual recurrence]
\label{lem:tail-square}
At any nonnegative reduced zero,
\[
  X_n=A_{2n}(\rho),\qquad 0\le n\le N-1.
\]
\end{lemma}

\begin{proof}
The residual convolution used here is the direct content of
\eqref{eq:terminal-tail-expanded} and \eqref{eq:eps-tail-expanded} in
Appendix~\ref{app:tail-square}.  Specifically,
\eqref{eq:terminal-tail-expanded} gives
\[
  \varepsilon_{N-1}=-X_1+\gamma_1,
\]
and \eqref{eq:eps-tail-expanded} gives, for \(1\le i\le N-2\) and
\(n=N-i\),
\[
  \varepsilon_i=\sum_{s=0}^{n}\gamma_sX_{n-s}.
\]
The derivation of these two displayed formulas is purely algebraic up to this
point: it uses System C, the A/B bridge, the tail normal form
\eqref{eq:B-tail-normal}, and the residual definitions.  It does not use the
positivity of \(K_i,a_i,b_i,c_i\), and it does not use the scalar inequality
Lemma~\ref{lem:scalar-A}.  The only place where terminal completion enters is
after these formulas have been derived, when
Lemma~\ref{lem:terminal-completion} allows us to set
\(\varepsilon_{N-1}=0\) and the reduced equations allow us to set
\(\varepsilon_i=0\) for \(0\le i\le N-2\).

Let
\[
  B_n=A_{2n}(\rho).
\]
Appendix~\ref{app:tail-square} verifies the comparison sequence by the
two explicit generating functions \eqref{eq:tail-square-B-generating} and
\eqref{eq:tail-square-gamma-generating}.  Their product is the identity
\eqref{eq:tail-square-generating-product}; equating coefficients gives
\eqref{eq:tail-square-A-convolution}, equivalently
\[
  \sum_{s=0}^{n}\gamma_sB_{n-s}=0,
  \qquad n\ge1 .
\]

Now \(X_0=1=B_0\).  The already proved terminal residual
\(\varepsilon_{N-1}=0\) gives
\[
  X_1=\gamma_1=A_2=B_1.
\]
For \(2\le n\le N-1\), the reduced residual equation
\(\varepsilon_{N-n}=0\) gives
\[
  \sum_{s=0}^{n}\gamma_sX_{n-s}=0.
\]
Since \(\gamma_0=-1\), this equation determines \(X_n\) uniquely from
\(X_0,\ldots,X_{n-1}\).  The sequence \(B_n=A_{2n}\) satisfies the same
recursion and has the same initial values.  Hence
\[
  X_n=B_n=A_{2n}(\rho),
  \qquad 0\le n\le N-1 .
\]
\end{proof}

For $1\le m\le N-2$ the preceding lemma gives
\[
  T_m^2=\rho^{-2N}A_{2N-2m-2}(\rho).
\]
The endpoint \(m=0\) is consistent with \(T_0=2N\).  Indeed, summing the
finite arithmetic-geometric series gives
\[
  A_{2N-2}(\rho)
  =
  \frac{2N\rho+2N+\rho^{2N}-1}{(1+\rho)^2}.
\]
At the GSW root, if \(R=\rho^{2N}\), then
\[
  R\{2N(1+\rho)+1\}=1 .
\]
Hence
\[
\begin{aligned}
  A_{2N-2}(\rho)
  &=
  \frac{2N(1+\rho)+R-1}{(1+\rho)^2}        \\
  &=
  \frac{\{2N(1+\rho)+1\}+R-2}{(1+\rho)^2}  \\
  &=
  \frac{R^{-1}+R-2}{(1+\rho)^2}
  =
  \frac{(1-R)^2}{R(1+\rho)^2}
  =
  4N^2R,
\end{aligned}
\]
because \(1-R=2N(1+\rho)R\).  Therefore
\(RT_0^2=4N^2R=A_{2N-2}(\rho)\), as required.

\subsection{Positivity of the cumulative margins}

\begin{lemma}[Scalar alternating-truncation inequality]
\label{lem:scalar-A}
For $1\le r<N$,
\[
  A_{2r-2}A_{2r}-A_{2r-1}^2>0
\]
at $\rho=\rho_N$.
\end{lemma}

\begin{proof}
First record the finite arithmetic-geometric closed form
\[
  A_L(\rho)
  =
  \sum_{s=0}^{L}(-1)^s(L+1-s)\rho^s
  =
  \frac{L+1+(L+2)\rho+(-\rho)^{L+2}}{(1+\rho)^2}.
\]
Indeed, with \(t=-\rho\),
\[
  \sum_{s=0}^{L}(L+1-s)t^s
  =
  (L+1)\sum_{s=0}^{L}t^s-\sum_{s=0}^{L}st^s
  =
  \frac{L+1-(L+2)t+t^{L+2}}{(1-t)^2}.
\]
Substituting \(t=-\rho\) gives the displayed formula.  Therefore
\[
\begin{aligned}
  A_{2r-2}
  &=\frac{2r-1+2r\rho+\rho^{2r}}{(1+\rho)^2},\\
  A_{2r-1}
  &=\frac{2r+(2r+1)\rho-\rho^{2r+1}}{(1+\rho)^2},\\
  A_{2r}
  &=\frac{2r+1+(2r+2)\rho+\rho^{2r+2}}{(1+\rho)^2}.
\end{aligned}
\]
A direct multiplication now gives
\[
\begin{aligned}
&(1+\rho)^4\bigl(A_{2r-2}A_{2r}-A_{2r-1}^2\bigr)\\
&\quad =
\bigl(2r-1+2r\rho+\rho^{2r}\bigr)
\bigl(2r+1+(2r+2)\rho+\rho^{2r+2}\bigr)\\
&\qquad
-\bigl(2r+(2r+1)\rho-\rho^{2r+1}\bigr)^2  \\
&\quad =
(1+\rho)^2\left\{\rho^{2r}(2r+1+2r\rho)-1\right\}.
\end{aligned}
\]
Hence
\[
  A_{2r-2}A_{2r}-A_{2r-1}^2
  =
  \frac{\rho^{2r}(2r+1+2r\rho)-1}{(1+\rho)^2}.
\]

For fixed \(\rho\in(0,1)\), define
\[
  F(x)=\log\bigl(\rho^{2x}(2x+1+2x\rho)\bigr).
\]
Then
\[
  F''(x)
  =
  -\frac{4(1+\rho)^2}{(2x+1+2x\rho)^2}<0.
\]
At the GSW root, \(F(N)=0\).  Also \(\rho_N>1/2\) gives
\[
  \rho_N^2(3+2\rho_N)>\frac14\cdot 4=1,
\]
so \(F(1)>0\).  By strict concavity, \(F(r)>0\) for every integer
\(1\le r<N\).  Therefore
\[
  \rho^{2r}(2r+1+2r\rho)-1>0,
\]
and the displayed determinant is positive.
\end{proof}

\begin{theorem}[Cumulative positivity]
\label{thm:K-positive}
Let $d\ge0$ satisfy $E_N(d)=0$.  Then
\[
  K_i(d)>0,\qquad 0\le i\le N-2.
\]
\end{theorem}
\begin{proof}
Since \(d\ge0\) and \(E_N(d)=0\), Lemma~\ref{lem:terminal-completion}
gives \(L_N(d)=0\).  The cumulative calculation below uses only this
terminal-completed identity, System C, the tail-square law
Lemma~\ref{lem:tail-square}, and the scalar inequality
Lemma~\ref{lem:scalar-A}.  It does not use the later A/B positivity
argument.  The corresponding expanded calculation in
Appendix~\ref{app:cumulative-margin} is organized by direct formula labels:
\eqref{eq:cumulative-K-prefix-form} gives the prefix form of \(K_i/R\),
\eqref{eq:cumulative-K-UV-form} rewrites it in the two tail variables
\(U,V\), \eqref{eq:cumulative-K-tail-square-form} inserts the tail-square
law, and \eqref{eq:cumulative-K-final-form} is the final scalar expression
used below.

From \(L_N(d)=0\) and the root relation,
\[
  \sum_{m=1}^{N-1}P_m=2N-q^{-1}.
\]
With the terminal convention \(P_N=q^{-1}\), this gives
\[
  T_m=2N-\Pi_m=\sum_{h=m+1}^NP_h,\qquad 0\le m\le N-1,
\]
and hence
\[
  P_m=T_{m-1}-T_m,\qquad 1\le m\le N.
\]

Fix \(0\le i\le N-2\) and put \(m=i+1\).  Write
\[
  D_i=\sum_{h=0}^id_h=P_m-1.
\]
By System C,
\[
  \sum_{h=0}^{m-1}\frac{c_h}{R}
  =
  \sum_{h=0}^{m-1}(P_h+\rho P_{h+1})
  =
  1+(1+\rho)\Pi_{m-1}+\rho P_m ,
\]
while the tail-sum identity gives
\[
  \sum_{j=m}^N\frac{c_j}{R}
  =
  T_{m-1}+\rho T_m
  =
  (1+\rho)(2N-\Pi_{m-1})-\rho P_m .
\]
Hence
\[
\begin{aligned}
  \frac{K_i}{R}
  &=
  1+(1+\rho)\Pi_{m-1}+\rho P_m                       \\
  &\quad
  -(P_m-1)\{(1+\rho)(2N-\Pi_{m-1})-\rho P_m\}          \\
  &=
  1+\rho P_m^2
  +(1+\rho)\{2N-P_m(2N-\Pi_{m-1})\}.
\end{aligned}
\]
Now set
\[
  U=T_{m-1}=2N-\Pi_{m-1},\qquad V=T_m=U-P_m .
\]
Using the root equation in the form
\[
  1+2N(1+\rho)=R^{-1},
\]
the preceding expression becomes
\[
\begin{aligned}
  \frac{K_i}{R}
  &=
  1+\rho(U-V)^2+(1+\rho)\{2N-(U-V)U\}                 \\
  &=
  R^{-1}-U^2+\rho V^2+(1-\rho)UV .
\end{aligned}
\]

Let
\[
  r=N-m.
\]
Since \(0\le i\le N-2\), we have \(1\le r\le N-1\).  By
Lemma~\ref{lem:tail-square},
\[
  U^2=R^{-1}A_{2r}(\rho),\qquad
  V^2=R^{-1}A_{2r-2}(\rho).
\]
The tails \(U,V\) are positive, so
\[
  UV=R^{-1}\sqrt{A_{2r}(\rho)A_{2r-2}(\rho)} .
\]
Therefore
\[
  \frac{K_i}{R}
  =
  R^{-1}\left\{
    1-A_{2r}+\rho A_{2r-2}
    +(1-\rho)\sqrt{A_{2r}A_{2r-2}}
  \right\}.
\]
The finite-sum identity
\begin{equation}
  A_{2r}-\rho A_{2r-2}-1=(1-\rho)A_{2r-1}
  \label{eq:K-proof-A-odd-reduction}
\end{equation}
is the identity \eqref{eq:A-odd-reduction} proved in
Appendix~\ref{app:cumulative-margin}.  That proof is a coefficient-level
finite-sum check from the definition of \(A_L\); it does not use the
cumulative-margin inequality, the positivity of \(K_i\), or the final
coefficient-positivity conclusion.  Thus
\[
  \frac{K_i}{R}
  =
  R^{-1}(1-\rho)
  \left\{
    \sqrt{A_{2r-2}A_{2r}}-A_{2r-1}
  \right\}.
\]
The tails \(U,V\) are positive, and Lemma~\ref{lem:tail-square} gives
\[
  A_{2r}(\rho)=RU^2>0,\qquad A_{2r-2}(\rho)=RV^2>0 .
\]
Thus
\[
  \sqrt{A_{2r-2}A_{2r}}-A_{2r-1}>0
\]
follows from
\[
  A_{2r-2}A_{2r}-A_{2r-1}^2>0,
\]
because the latter implies
\[
  \sqrt{A_{2r-2}A_{2r}}>|A_{2r-1}|\ge A_{2r-1}.
\]
The strict determinant inequality is exactly Lemma~\ref{lem:scalar-A}.
Hence \(K_i>0\) for every \(0\le i\le N-2\).
\end{proof}

\subsection{Positivity of \texorpdfstring{$b$ and $a$}{b and a}}

\begin{lemma}[A/B bridge]
\label{lem:ab-bridge}
For $0\le i\le N-2$,
\[
  \rho a_i-b_i=\rho d_i c_{i+1}.
\]
\end{lemma}

\begin{proof}
Subtract the recurrence \eqref{eq:Bpen} from $\rho$ times
\eqref{eq:Apen} for $i=N-2$, and subtract \eqref{eq:Brec} from $\rho$ times
\eqref{eq:Arec} for the remaining indices.  All terms cancel except
$\rho d_i c_{i+1}$.
\end{proof}

\begin{corollary}[Coefficient positivity]
\label{cor:ab-positive}
At the positive reduced zero $d^\star$, all coefficients $a_i,b_i,c_i,d_i$
are strictly positive.
\end{corollary}

\begin{proof}
Theorem~\ref{thm:positive-reduced-zero} gives $d_i^\star>0$.
Lemma~\ref{lem:c-positive} gives $c_i(d^\star)>0$.
Theorem~\ref{thm:K-positive} and Lemma~\ref{lem:cumulative} give
\[
  b_i(d^\star)=\frac{\rho}{1-\rho}K_i(d^\star)>0,
\]
because $0<\rho<1$.  Finally Lemma~\ref{lem:ab-bridge} gives
\[
  a_i(d^\star)
  =
  \frac{b_i(d^\star)+\rho d_i^\star c_{i+1}(d^\star)}{\rho}>0.
\]
Thus $a_i(d^\star)>0$ for $0\le i\le N-2$.  It remains to treat the terminal coefficient
$a_{N-1}$.  By System A,
\[
  a_{N-1}
  =
  1-c_N\left(1+\sum_{k=0}^{N-2}d_k^\star\right)
  =
  1-qP_{N-1}.
\]
At the same terminal-completed reduced zero, Lemma~\ref{lem:tail-square}
gives
\[
  RT_{N-2}^2=A_2,\qquad RT_{N-1}^2=A_0=1.
\]
Since $T_{N-1}=q^{-1}>0$ and
$T_{m-1}-T_m=P_m>0$ for $1\le m\le N-1$, all earlier tails are positive.
Thus the square law selects the positive roots
$T_{N-1}=q^{-1}$ and $T_{N-2}=q^{-1}\sqrt{A_2}$.  Hence
\[
  P_{N-1}=T_{N-2}-T_{N-1}
  =
  q^{-1}(\sqrt{A_2}-1).
\]
Here
\[
  A_2=3-2\rho+\rho^2.
\]
Therefore
\[
  a_{N-1}=1-qP_{N-1}=2-\sqrt{3-2\rho+\rho^2}.
\]
Since $0<\rho<1$, one has $3-2\rho+\rho^2<4$, and so
$a_{N-1}>0$.
\end{proof}

\section{Certificate upper bound and minimax value}

\begin{theorem}[GSW low-rank certificate existence]
\label{thm:main}
For every $N\ge3$, let $\rho_N$ be the root of
\[
  \rho_N^{2N}(2N\rho_N+2N+1)=1.
\]
Then there exist vectors
\[
  d\in\mathbb{R}_{>0}^{N-1},\qquad
  a\in\mathbb{R}_{>0}^N,\qquad
  b\in\mathbb{R}_{>0}^{N-1},\qquad
  c\in\mathbb{R}_{>0}^{N+1}
\]
satisfying Systems C, A, B, and D, namely
\[
  \varepsilon_0=\varepsilon_1=\cdots=\varepsilon_N=0.
\]
Consequently the GSW low-rank PEP certificate exists for every $N\ge3$.
\end{theorem}

\begin{proof}
Theorem~\ref{thm:positive-reduced-zero} supplies a positive $d^\star$ with
$E_N(d^\star)=0$, i.e. $\varepsilon_0=\cdots=\varepsilon_{N-2}=0$.
Lemma~\ref{lem:terminal-completion} gives
$\varepsilon_{N-1}=\varepsilon_N=0$ at the same $d^\star$.
Corollary~\ref{cor:ab-positive} proves positivity of every coefficient in
$a_i,b_i,c_i,d_i$.  Hence all GSW systems and residual equations hold with
strictly positive coefficients.
\end{proof}

\begin{corollary}[GSW strengthened certificate assertion]
\label{cor:gsw-strengthened}
In the normalization and notation of Section~\ref{sec:system}, the structured
low-rank certificate asserted by Grimmer--Shu--Wang exists for every
$N\ge3$.
\end{corollary}

\begin{proof}
The multiplier pattern is the one displayed in
Proposition~\ref{prop:cert-rate}: the nonzero weights are
$\lambda_{\star j}=c_j$, $\lambda_{i,i+1}=a_i$,
$\lambda_{i+1,i}=b_i$, and $\lambda_{ij}=d_i c_j$ for $i+2\le j$.
Theorem~\ref{thm:main} gives positive $a,b,c,d$ satisfying Systems C, A, B,
and D, equivalently all residual equations.  These are precisely the
low-rank certificate equations stated above.
\end{proof}

\begin{corollary}[Drori--Teboulle upper bound]
\label{cor:upper-bound}
For the constant step $\alpha=1+\rho_N$,
\[
  f(x_N)-f_\star
  \le
  \frac{\rho_N^{2N}}{2}\|x_0-x_\star\|^2
  =
  \frac{1}{2(2N\alpha+1)}\|x_0-x_\star\|^2 .
\]
\end{corollary}

\begin{proof}
Apply Proposition~\ref{prop:cert-rate} to the certificate from
Theorem~\ref{thm:main}.
\end{proof}

\subsection{One-dimensional matching lower bounds}
\label{sec:minimax}

Fix $N\ge3$.  For a real constant step $\alpha$, let
\[
  \mathcal W_N(\alpha)
  =
  \sup\left\{
    f(x_N)-f_\star:
    f \text{ convex and }1\text{-smooth},\;
    \|x_0-x_\star\|^2\le1,\;
    x_{k+1}=x_k-\alpha \nabla f(x_k)
  \right\}.
\]
The following lower bounds are the one-dimensional examples used to match the
certificate upper bound.

\begin{lemma}[Quadratic lower bound]
\label{lem:quadratic-lower}
For every $\alpha\in\mathbb R$,
\[
  \mathcal W_N(\alpha)\ge \frac12(1-\alpha)^{2N}.
\]
\end{lemma}

\begin{proof}
Take $f(x)=x^2/2$ on the real line, with minimizer $x_\star=0$ and
$x_0=1$.  The gradient descent recurrence is
\[
  x_{k+1}=x_k-\alpha x_k=(1-\alpha)x_k,
\]
so $x_N=(1-\alpha)^N$.  Hence
\[
  f(x_N)-f_\star=\frac12(1-\alpha)^{2N}.
\]
This example has $\|x_0-x_\star\|^2=1$.
\end{proof}

\begin{lemma}[Huber lower bound]
\label{lem:huber-lower}
For every $\alpha\ge0$,
\[
  \mathcal W_N(\alpha)\ge \frac{1}{2(2N\alpha+1)}.
\]
\end{lemma}

\begin{proof}
Let
\[
  \delta=\frac{1}{2N\alpha+1}.
\]
When $\alpha=0$, this gives $\delta=1$ and the trajectory below is the
constant trajectory $x_k=1$; thus the endpoint case is included in the same
formula.
Define the one-dimensional Huber function
\[
  H_\delta(x)=
  \begin{cases}
    x^2/2, & |x|\le\delta,\\[2mm]
    \delta |x|-\delta^2/2, & |x|\ge\delta .
  \end{cases}
\]
It is convex, differentiable, and $1$-smooth: its derivative is the clipped
map
\[
  H_\delta'(x)=
  \begin{cases}
    -\delta, & x\le-\delta,\\
    x, & |x|\le\delta,\\
    \delta, & x\ge\delta,
  \end{cases}
\]
which is nondecreasing and $1$-Lipschitz.  Its minimizer is $x_\star=0$.

Start from $x_0=1$.  We claim that
\[
  x_k=1-k\alpha\delta,\qquad 0\le k\le N.
\]
Indeed, if $0\le k\le N$, then
\[
  x_k
  =
  1-\frac{k\alpha}{2N\alpha+1}
  \ge
  1-\frac{N\alpha}{2N\alpha+1}
  =
  \frac{N\alpha+1}{2N\alpha+1}
  \ge
  \frac{1}{2N\alpha+1}
  =\delta .
\]
Thus every iterate lies in the right linear region of $H_\delta$, where
$H_\delta'(x_k)=\delta$.  Therefore
$x_{k+1}=x_k-\alpha\delta$, proving the claimed formula by induction.

The final iterate is
\[
  x_N=\frac{N\alpha+1}{2N\alpha+1}.
\]
Since $x_N\ge\delta$,
\[
  H_\delta(x_N)-H_\delta(0)
  =
  \delta x_N-\frac{\delta^2}{2}
  =
  \frac{N\alpha+1}{(2N\alpha+1)^2}
  -\frac{1}{2(2N\alpha+1)^2}
  =
  \frac{1}{2(2N\alpha+1)}.
\]
Again $\|x_0-x_\star\|^2=1$.
\end{proof}

\begin{corollary}[Constant-step lower envelope]
\label{cor:lower-envelope}
For every nonnegative $\alpha$,
\[
  \mathcal W_N(\alpha)
  \ge
  \max\left\{
    \frac12(1-\alpha)^{2N},
    \frac{1}{2(2N\alpha+1)}
  \right\}.
\]
\end{corollary}

\begin{proof}
Take the larger of the two lower bounds from
Lemmas~\ref{lem:quadratic-lower} and~\ref{lem:huber-lower}.
\end{proof}

\begin{lemma}[Balancing of the lower envelope]
\label{lem:balance}
Let $\rho_N$ be the root of \eqref{eq:rho-root} and set
$\alpha_\star=1+\rho_N$.  Then
\[
  \frac12(\alpha_\star-1)^{2N}
  =
  \frac{1}{2(2N\alpha_\star+1)}
  =
  \frac{\rho_N^{2N}}{2}.
\]
Moreover,
\[
  \max\left\{
    \frac12(1-\alpha)^{2N},
    \frac{1}{2(2N\alpha+1)}
  \right\}
  \ge
  \frac{\rho_N^{2N}}{2}
  \qquad(\alpha\ge0),
\]
with equality at $\alpha=\alpha_\star$.
\end{lemma}

\begin{proof}
The equality at $\alpha_\star$ is exactly
\[
  \rho_N^{2N}(2N(1+\rho_N)+1)=1,
\]
which is \eqref{eq:rho-root}.  It remains to show that this crossing is the
minimum of the displayed maximum.

On $0\le\alpha\le1$, the Huber term
$1/(2(2N\alpha+1))$ is decreasing in $\alpha$, so it is at least its value at
$\alpha=1$.  Since $\alpha_\star>1$,
\[
  \frac{1}{2(2N+1)}
  >
  \frac{1}{2(2N\alpha_\star+1)}
  =
  \frac{\rho_N^{2N}}{2}.
\]
Thus the maximum is strictly larger than the balanced value on
$0\le\alpha\le1$.

On $\alpha\ge1$, the quadratic term
$(\alpha-1)^{2N}/2$ is increasing and the Huber term is decreasing.  For
$1\le\alpha\le\alpha_\star$, the maximum is at least the Huber term, hence at
least its value at $\alpha_\star$.  For $\alpha\ge\alpha_\star$, the maximum
is at least the quadratic term, hence at least its value at $\alpha_\star$.
This proves the claim.
\end{proof}

\begin{lemma}[Negative steps are not competitive]
\label{lem:negative-steps}
Let $\rho_N$ be the root of \eqref{eq:rho-root} and set
$r_\star=\rho_N^{2N}/2$.  If $\alpha<0$, then
\[
  \mathcal W_N(\alpha)>r_\star .
\]
\end{lemma}

\begin{proof}
By Lemma~\ref{lem:quadratic-lower},
\[
  \mathcal W_N(\alpha)\ge \frac12(1-\alpha)^{2N}.
\]
For $\alpha<0$, one has $1-\alpha>1$, hence
\[
  \mathcal W_N(\alpha)>\frac12.
\]
Since $0<\rho_N<1$, $r_\star=\rho_N^{2N}/2<1/2$.  Thus
$\mathcal W_N(\alpha)>r_\star$ for every negative step.
\end{proof}

\begin{theorem}[Drori--Teboulle minimax constant-step value]
\label{thm:dt-minimax}
For $N\ge3$, let $\rho_N\in(0,1)$ satisfy
\[
  \rho_N^{2N}(2N\rho_N+2N+1)=1,
\]
and set
\[
  \alpha_\star=1+\rho_N,\qquad
  r_\star=\frac{\rho_N^{2N}}{2}
  =
  \frac{1}{2(2N\alpha_\star+1)}.
\]
Then
\[
  \inf_{\alpha\in\mathbb R}\mathcal W_N(\alpha)=r_\star,
\]
and the infimum is attained at $\alpha=\alpha_\star$.  In particular, the
constant step $\alpha_\star$ is minimax optimal and its worst-case value is
$r_\star$.
\end{theorem}

\begin{proof}
Corollary~\ref{cor:upper-bound} gives
$\mathcal W_N(\alpha_\star)\le r_\star$.  Corollary~\ref{cor:lower-envelope}
and Lemma~\ref{lem:balance} give
$\mathcal W_N(\alpha)\ge r_\star$ for every $\alpha\ge0$, while
Lemma~\ref{lem:negative-steps} gives the same strict lower bound for every
$\alpha<0$.  Therefore
$\inf_{\alpha\in\mathbb R}\mathcal W_N(\alpha)=r_\star$, and the value is
attained at $\alpha_\star$.
\end{proof}

\appendix

\section{Coefficient matching for the dual certificate}
\label{app:dual-coeff}

This appendix expands the coefficient verification used in
Proposition~\ref{prop:cert-rate}.  The interpolation inequalities are the
standard smooth-convex PEP inequalities of Drori--Teboulle and of Taylor,
Hendrickx, and Glineur~\cite{drori-teboulle,taylor-hendrickx-glineur}.  The
low-rank multiplier pattern is the strengthened GSW ansatz
\cite{grimmer-shu-wang}.  What is checked here is the normalized coefficient
identity used in this paper: after the GSW pattern is inserted, the remaining
coefficient equations are exactly Systems C, A, B, and D.

Put
\[
  z=x_0-x_\star,
  \qquad
  y_i=\langle g_i,z\rangle,
  \qquad
  G_{ij}=\langle g_i,g_j\rangle,
  \qquad
  v=\sum_{i=0}^Nc_i g_i .
\]
Also write
\[
  H_i=\sum_{\ell=0}^{i-1}g_\ell,
  \qquad H_0=0,
  \qquad x_i=x_0-\alpha H_i .
\]
Thus, if \(i<j\), then
\[
  x_i-x_j=\alpha\sum_{\ell=i}^{j-1}g_\ell,
\]
and if \(i>j\), then
\[
  x_i-x_j=-\alpha\sum_{\ell=j}^{i-1}g_\ell .
\]
These two elementary identities are the only place where the gradient-descent
trajectory is used in the coefficient matching.

For reference, the star-row slack has the explicit form
\begin{equation}
  Q_{\star j}
  =
  f_\star-f_j+y_j
  -\alpha\sum_{\ell=0}^{j-1}G_{j\ell}
  -\frac12G_{jj},
  \qquad 0\le j\le N,
  \label{eq:dual-coeff-star-row}
\end{equation}
because \(g_\star=0\) and
\(x_\star-x_j=-z+\alpha H_j\).  For an adjacent forward row,
\begin{equation}
  Q_{i,i+1}
  =
  f_i-f_{i+1}
  -\alpha G_{i+1,i}
  -\frac12G_{ii}+G_{i,i+1}-\frac12G_{i+1,i+1},
  \label{eq:dual-coeff-forward-slack}
\end{equation}
and for an adjacent reversed row,
\begin{equation}
  Q_{i+1,i}
  =
  f_{i+1}-f_i
  +\alpha G_{ii}
  -\frac12G_{i+1,i+1}+G_{i+1,i}-\frac12G_{ii}.
  \label{eq:dual-coeff-reversed-slack}
\end{equation}
For a nonadjacent row \(i+2\le j\),
\begin{equation}
  Q_{ij}
  =
  f_i-f_j
  -\alpha\sum_{\ell=i}^{j-1}G_{j\ell}
  -\frac12G_{ii}+G_{ij}-\frac12G_{jj}.
  \label{eq:dual-coeff-nonadj-slack}
\end{equation}
The quadratic term on the right side of \eqref{eq:dual-identity} contributes
\[
  \frac{1}{4r}\|v\|^2
  =
  \frac{1}{2R}\sum_{p,q=0}^Nc_pc_qG_{pq},
\]
because \(r=R/2\).  Equations
\eqref{eq:dual-coeff-star-row}--\eqref{eq:dual-coeff-nonadj-slack}, together
with this last identity, give a direct recipe for every coefficient of
\[
  \mathcal D=
  \sum_{i,j}\lambda_{ij}Q_{ij}
  -\left(
  f_\star-f_N+\langle v,x_0-x_\star\rangle-\frac{1}{4r}\|v\|^2
  +\sum_{i=0}^{N-1}\varepsilon_i(f_i-f_\star)
  +\frac{\varepsilon_N}{2}\|g_0\|^2
  \right).
\]
We now list the coefficient families.  The displayed Gram rows use the same
directed-slot bookkeeping as the GSW ansatz: adjacent slots \(G_{i,i+1}\) and
\(G_{i+1,i}\) are first collected separately, and only after both directed
coefficients vanish do we impose the symmetric Gram convention \(G_{ij}=G_{ji}\).

\paragraph{Function-value coefficients.}
Since each \(Q_{ij}\) contains \(f_i-f_j\), the coefficient of
\(f_i-f_\star\) is ``outgoing weight minus incoming weight'', with the residual
subtracted.  Thus, for \(0\le i\le N-1\),
\begin{equation}
  [f_i-f_\star]\mathcal D
  =
  \sum_j\lambda_{ij}-\sum_j\lambda_{ji}-\varepsilon_i,
  \label{eq:dual-coeff-f-interior}
\end{equation}
whereas the coefficient of \(f_N-f_\star\) has the additional \(+1\) coming
from subtracting \(f_\star-f_N\):
\begin{equation}
  [f_N-f_\star]\mathcal D
  =
  \sum_j\lambda_{Nj}-\sum_j\lambda_{jN}+1 .
  \label{eq:dual-coeff-f-terminal}
\end{equation}
Inserting the low-rank pattern gives the residual equations one by one:
\[
\begin{aligned}
  [f_0-f_\star]\mathcal D
  &=a_0+d_0\sum_{j=2}^Nc_j-b_0-c_0-\varepsilon_0,\\[1mm]
  [f_i-f_\star]\mathcal D
  &=b_{i-1}+a_i+d_i\sum_{j=i+2}^Nc_j
    -a_{i-1}-b_i
    -c_i\left(1+\sum_{k=0}^{i-2}d_k\right)-\varepsilon_i,
    \quad 1\le i\le N-2,\\[1mm]
  [f_{N-1}-f_\star]\mathcal D
  &=b_{N-2}+a_{N-1}-a_{N-2}
    -c_{N-1}\left(1+\sum_{k=0}^{N-3}d_k\right)-\varepsilon_{N-1},\\[1mm]
  [f_N-f_\star]\mathcal D
  &=1-a_{N-1}-c_N\left(1+\sum_{k=0}^{N-2}d_k\right).
\end{aligned}
\]
The first three rows are exactly System D,
\eqref{eq:D0}--\eqref{eq:Dterm1}; the last row is the negative of the terminal
A-equation \eqref{eq:Aterm}.  Hence all function-value coefficients vanish
when Systems A and D hold.

\paragraph{The \(y_i\) coefficients.}
Only the star rows and the linear term \(\langle v,z\rangle\) contain \(y_i\).
By \eqref{eq:dual-coeff-star-row}, the star row contributes \(c_i y_i\); the
linear term contributes the same \(c_i y_i\) and is subtracted in \(\mathcal D\).
Thus
\[
  [y_i]\mathcal D=c_i-c_i=0,
  \qquad 0\le i\le N.
\]
For clarity, the star-row choices are
\begin{center}
\begingroup
\small
\setlength{\tabcolsep}{3pt}
\renewcommand{\arraystretch}{1.12}
\begin{tabular}{@{}p{0.16\textwidth}|p{0.50\textwidth}|p{0.24\textwidth}@{}}
\hline
\textbf{index} &
\(\left[y_i\right]\sum_{p,q}\lambda_{pq}Q_{pq}\) &
\(\left[y_i\right]\langle v,z\rangle\) \\
\hline
\(0\le i\le N-2\) &
\(\displaystyle \lambda_{\star i}=c_i
=R\left(\alpha\sum_{k=0}^id_k-d_i+\alpha\right)\) &
\(c_i\) \\
\(i=N-1\) &
\(\displaystyle \lambda_{\star,N-1}=c_{N-1}
=R\left(1+\sum_{k=0}^{N-2}d_k+\frac{\rho}{q}\right)\) &
\(c_{N-1}\) \\
\(i=N\) &
\(\lambda_{\star N}=c_N=q\) &
\(c_N\) \\
\hline
\end{tabular}
\endgroup
\end{center}
System C is therefore not needed for the cancellation \([y_i]\mathcal D=0\)
itself.  Its role is to provide the affine parametrization of the \(c_i\)'s
shown in the middle column; those formulas are used in the Gram-coefficient
rows below.

\paragraph{Adjacent off-diagonal Gram coefficients.}
The adjacent coefficients are the only off-diagonal rows where the adjacent
multipliers \(a_i\) and \(b_i\) appear.  We first keep the two ordered slots
\(G_{i,i+1}\) and \(G_{i+1,i}\) separate as a bookkeeping device.  This is useful
because the trajectory term \(-\langle g_j,x_i-x_j\rangle\) is directed, while
the norm term is symmetric.  Since Systems A and B make the two directed
coefficients zero separately, imposing the symmetric Gram convention
\(G_{ij}=G_{ji}\) afterward loses no factor of \(2\).

With the low-rank weights inserted, the forward adjacent slots are
\begin{align}
  [G_{i,i+1}]\mathcal D
  &=
  \alpha a_i-\frac{c_{i+1}^2}{R}-\frac{c_i c_{i+1}}{R}+a_{i+1}
  +(1+\alpha)c_{i+1}\left(1+\sum_{k=0}^{i-1}d_k\right) \notag\\
  &\quad
  +d_{i+1}\sum_{j=i+3}^Nc_j-(2\alpha-1)b_{i+1},
  \qquad 0\le i\le N-3,                                  \label{eq:dual-coeff-forward-adj-int}\\
  [G_{N-2,N-1}]\mathcal D
  &=
  \alpha a_{N-2}-\frac{c_{N-1}^2}{R}-\frac{c_{N-2}c_{N-1}}{R}
  +a_{N-1} \notag\\
  &\quad
  +(1+\alpha)c_{N-1}\left(1+\sum_{k=0}^{N-3}d_k\right),       \label{eq:dual-coeff-forward-adj-pen}\\
  [G_{N-1,N}]\mathcal D
  &=
  a_{N-1}+c_N\left(1+\sum_{k=0}^{N-2}d_k\right)-1 .           \label{eq:dual-coeff-forward-adj-term}
\end{align}
These three rows are exactly the left sides of
\eqref{eq:Arec}, \eqref{eq:Apen}, and \eqref{eq:Aterm} after every term is
moved to one side.  Thus the forward adjacent coefficients vanish by System A.

The reversed adjacent slots are
\begin{align}
  [G_{i+1,i}]\mathcal D
  &=
  \alpha b_i-\frac{\rho c_{i+1}^2}{R}+\frac{c_i c_{i+1}}{R}
  +\rho a_{i+1}
  -c_{i+1}\left(1+\sum_{k=0}^{i-1}d_k\right) \notag\\
  &\quad
  +\rho d_{i+1}\sum_{j=i+3}^Nc_j-\rho(2\alpha-1)b_{i+1},
  \qquad 0\le i\le N-3,                                      \label{eq:dual-coeff-rev-adj-int}\\
  [G_{N-1,N-2}]\mathcal D
  &=
  \alpha b_{N-2}-\frac{\rho c_{N-1}^2}{R}
  +\frac{c_{N-2}c_{N-1}}{R}
  +\rho a_{N-1}
  -c_{N-1}\left(1+\sum_{k=0}^{N-3}d_k\right).                \label{eq:dual-coeff-rev-adj-pen}
\end{align}
These are the left sides of \eqref{eq:Brec} and \eqref{eq:Bpen}; hence the
reversed adjacent coefficients vanish by System B.

\paragraph{Diagonal Gram coefficients.}
Diagonal entries are not doubled.  After the \(y_i\) coefficients have been
cancelled and the System C formulas for \(c\) have been substituted, the
diagonal rows are as follows.  The factor \(1/2\) is the usual diagonal factor
from the terms \(-\frac12\|g_p-g_q\|^2\) and
\(\frac{1}{2R}\sum_{p,q}c_pc_qG_{pq}\).
\begin{center}
\begingroup
\scriptsize
\setlength{\tabcolsep}{2pt}
\renewcommand{\arraystretch}{1.35}
\begin{tabular}{@{}p{0.11\textwidth}|p{0.66\textwidth}|p{0.17\textwidth}@{}}
\hline
\text{range} & \([G_{kk}]\mathcal D\) & \text{identity used} \\
\hline
\(k=0\) &
\[
\frac12\left(
-c_0-a_0-d_0\sum_{j=2}^Nc_j
+(2\alpha-1)b_0+\frac{c_0^2}{R}
-\varepsilon_N
\right)
\]
&
\eqref{eq:Dterm2} \\[-1mm]
\hline
\(1\le k\le N-2\) &
\[
\frac12\left(
\begin{aligned}
&\alpha a_{k-1}-\frac{c_k^2}{R}
-\frac{c_{k-1}c_k}{R}+a_k  \\
&\quad +(1+\alpha)c_k
  \left(1+\sum_{h=0}^{k-2}d_h\right)
+d_k\sum_{j=k+2}^Nc_j
-(2\alpha-1)b_k
\end{aligned}
\right)
\]
&
\eqref{eq:Arec} with \(i=k-1\) \\[-1mm]
\hline
\(k=N-1\) &
\[
\frac12\left(
\begin{aligned}
&\alpha a_{N-2}-\frac{c_{N-1}^2}{R}
-\frac{c_{N-2}c_{N-1}}{R}+a_{N-1} \\
&\quad +(1+\alpha)c_{N-1}
\left(1+\sum_{h=0}^{N-3}d_h\right)
\end{aligned}
\right)
\]
&
\eqref{eq:Apen} \\[-1mm]
\hline
\(k=N\) &
\[
\frac12\left(
a_{N-1}
+c_N\left(1+\sum_{h=0}^{N-2}d_h\right)-1
\right)
\]
&
\eqref{eq:Aterm} \\
\hline
\end{tabular}
\endgroup
\end{center}
Therefore the diagonal Gram coefficients vanish by System D at \(k=0\) and by
the diagonal parts of System A for \(1\le k\le N\).

\paragraph{Nonadjacent off-diagonal coefficients.}
For \(i+2\le j\), the low-rank ansatz contains the multiplier
\(\lambda_{ij}=d_ic_j\).  The coefficient of the corresponding nonadjacent
formal slot is
\[
  [G_{ij}]\mathcal D=\lambda_{ij}-d_ic_j=0.
\]
This is the basic low-rank cancellation inherited from the GSW pattern: the
nonadjacent multiplier is chosen to cancel exactly the product term created by
expanding the star-row and trajectory contributions.  Every nonadjacent
pattern not listed in the ansatz has multiplier zero and no compensating
product term, so its coefficient is also zero.

\paragraph{Exhaustion of coefficient classes.}
The variables appearing in \(\mathcal D\) are only
\(f_i-f_\star\), \(y_i\), and Gram entries \(G_{ij}\).  The preceding paragraphs
have covered all of them: function values, point-gradient terms, adjacent
off-diagonal Gram slots, diagonal Gram slots, and nonadjacent off-diagonal Gram
slots.  The summary is:
\begin{center}
\begingroup
\scriptsize
\setlength{\tabcolsep}{2.2pt}
\renewcommand{\arraystretch}{1.12}
\begin{tabular}{@{}p{0.20\textwidth}|p{0.43\textwidth}|p{0.29\textwidth}@{}}
\hline
\text{coefficient family}
& \text{vanishing expression}
& \text{source} \\
\hline
\(f_i-f_\star,\ 0\le i\le N\)
&
Outgoing minus incoming weight, with residuals and the terminal \(+1\) row.
&
\(\eqref{eq:D0}\)--\(\eqref{eq:Dterm1}\), \(\eqref{eq:Aterm}\) \\
\hline
\(y_i\)
&
\(\lambda_{\star i}-c_i=0\)
&
\(\lambda_{\star i}=c_i\), \(v=\sum_i c_ig_i\) \\
\hline
\(G_{i,i+1}\)
&
Forward adjacent rows
\(\eqref{eq:dual-coeff-forward-adj-int}\)--\(\eqref{eq:dual-coeff-forward-adj-term}\).
&
\(\eqref{eq:Arec}\), \(\eqref{eq:Apen}\), \(\eqref{eq:Aterm}\) \\
\hline
\(G_{i+1,i}\)
&
Reversed adjacent rows
\(\eqref{eq:dual-coeff-rev-adj-int}\)--\(\eqref{eq:dual-coeff-rev-adj-pen}\).
&
\(\eqref{eq:Brec}\), \(\eqref{eq:Bpen}\) \\
\hline
\(G_{kk}\)
&
Diagonal rows displayed above.
&
\(\eqref{eq:Dterm2}\), \(\eqref{eq:Arec}\), \(\eqref{eq:Apen}\), \(\eqref{eq:Aterm}\) \\
\hline
\(G_{ij},\ i+2\le j\)
&
\(\lambda_{ij}-d_ic_j=0\)
&
\(\lambda_{ij}=d_ic_j\) \\
\hline
Other off-diagonal slots
&
Zero multiplier and no low-rank product term.
&
No listed low-rank weight \\
\hline
\end{tabular}
\endgroup
\end{center}

Thus every coefficient of \(\mathcal D\) is zero.  This proves
\eqref{eq:dual-identity}.

\section{Tail-prefix proof of the lower-face barrier}
\label{app:lower-face}

The proof of Lemma~\ref{lem:lower-faces} uses a tail-prefix normal form for
the coefficient vectors.  The purpose of this appendix is only to rewrite
Systems C, A, B, and D in variables adapted to the lower coordinate faces.
No reduced equation \(E_N(d)=0\), no terminal residual equation, and no
positivity conclusion from later sections is used here.

Define
\[
  P_0=1,\qquad
  P_m=1+\sum_{k=0}^{m-1}d_k\quad(1\le m\le N-1),
  \qquad P_N=\rho^{-N}.
\]
The last value is a terminal boundary convention; there is no variable
\(d_{N-1}\).  Let
\[
  T_m=\sum_{h=m+1}^NP_h\quad(0\le m\le N-1),\qquad T_N:=0,
\]
and put
\[
  R=\rho^{2N},\qquad u=1-\rho,\qquad v=2\rho-1.
\]
Then
\[
  P_m=T_{m-1}-T_m\quad(1\le m\le N-1),
  \qquad P_N=T_{N-1}-T_N,\qquad T_{N-1}=P_N=\rho^{-N}.
\]

We first rewrite System C.  For \(0\le i\le N-2\),
\[
\begin{aligned}
  P_i+\rho P_{i+1}
  &=
  1+\sum_{k=0}^{i-1}d_k
  +\rho\left(1+\sum_{k=0}^{i}d_k\right)        \\
  &=
  \alpha\sum_{k=0}^{i}d_k-d_i+\alpha ,
\end{aligned}
\]
because \(\alpha=1+\rho\).  For \(i=N-1\),
\[
  P_{N-1}+\rho P_N
  =
  1+\sum_{k=0}^{N-2}d_k+\frac{\rho}{\rho^N}
  =
  1+\sum_{k=0}^{N-2}d_k+\frac{\rho}{q}.
\]
Also \(RP_N=\rho^{2N}\rho^{-N}=\rho^N=q\).  Hence System C is equivalently
\[
  c_i=R(P_i+\rho P_{i+1})\quad(0\le i\le N-1),
  \qquad c_N=RP_N.
\]
Consequently, for \(1\le k\le N\),
\begin{equation}
  \sum_{j=k}^Nc_j=R(T_{k-1}+\rho T_k).
  \label{eq:lower-tail-c-sum}
\end{equation}
Indeed,
\[
\begin{aligned}
  \sum_{j=k}^Nc_j
  &=
  R\sum_{j=k}^{N-1}(P_j+\rho P_{j+1})+RP_N  \\
  &=
  R\left\{
    P_k+(1+\rho)\sum_{h=k+1}^{N}P_h
  \right\}
  =
  R(T_{k-1}+\rho T_k).
\end{aligned}
\]
For \(k=N\), this reads \(c_N=R(T_{N-1}+\rho T_N)=RP_N\), so the terminal
case is included by the convention \(T_N=0\).

Normalize
\[
  B_i=\frac{b_i}{\rho R}\qquad(0\le i\le N-2).
\]
We claim that
\begin{equation}
  B_i
  =
  T_iT_{i+1}+(\rho-2)T_{i+1}^2
  +2u^2\sum_{\ell=i+2}^{N-1}v^{\ell-i-2}T_\ell^2 .
  \label{eq:lower-B-tail-normal}
\end{equation}

We prove \eqref{eq:lower-B-tail-normal} by backward induction.  First take
\(i=N-2\).  Put
\[
  X=T_{N-2},\qquad Y=T_{N-1}=P_N,\qquad A=P_{N-2}.
\]
Then
\[
  P_{N-1}=X-Y,\qquad RY^2=1,
\]
and
\[
  \frac{c_{N-1}}{R}=P_{N-1}+\rho P_N=X-uY,
\]
\[
  \frac{c_{N-2}}{R}=P_{N-2}+\rho P_{N-1}=A+\rho(X-Y).
\]
The terminal equation of System A gives
\[
  a_{N-1}=1-c_NP_{N-1}=1-RY(X-Y).
\]
Using the terminal equation of System B,
\[
  b_{N-2}
  =
  \frac{
    \rho c_{N-1}^2/R-c_{N-2}c_{N-1}/R-\rho a_{N-1}
    +c_{N-1}P_{N-2}
  }{1+\rho},
\]
we compute the numerator after division by \(R\).  With
\[
  C=X-uY=\frac{c_{N-1}}{R},
  \qquad
  D=A+\rho(X-Y)=\frac{c_{N-2}}{R},
\]
we have
\[
\begin{aligned}
&\frac1R\left(
\rho c_{N-1}^2/R
-\frac{c_{N-2}c_{N-1}}{R}
-\rho a_{N-1}
+c_{N-1}A
\right)                                                    \\
&\qquad =
\rho C^2-DC-\frac{\rho}{R}
+\rho Y(X-Y)+AC.
\end{aligned}
\]
Since \(RY^2=1\) and \(D=A+\rho(X-Y)\), this becomes
\[
\begin{aligned}
&\rho C^2-\rho(X-Y)C-\rho Y^2+\rho Y(X-Y)       \\
&\qquad =
\rho\{C^2-(X-Y)C+Y(X-Y)-Y^2\}.
\end{aligned}
\]
Substituting \(C=X-uY\) and \(u=1-\rho\), the last expression simplifies to
\[
  \rho(1+\rho)\{XY+(\rho-2)Y^2\}.
\]
Therefore
\[
  b_{N-2}
  =
  \rho R\{XY+(\rho-2)Y^2\},
\]
and hence
\[
  B_{N-2}
  =
  \frac{b_{N-2}}{\rho R}
  =
  T_{N-2}T_{N-1}+(\rho-2)T_{N-1}^2.
\]
This is \eqref{eq:lower-B-tail-normal} with an empty tail sum.

For the induction step, fix \(0\le i\le N-3\).  We first remove
\(a_{i+1}\) from the \(i\)-th equation of System B.  The local A/B
bridge at index \(i+1\), obtained by subtracting the \((i+1)\)-st
equation of System B from \(\rho\) times the corresponding equation of
System A, is
\[
  \rho a_{i+1}-b_{i+1}=\rho d_{i+1}c_{i+2}.
\]
This identity uses the recursive A/B equations when \(i+1\le N-3\), and
the terminal A/B equations when \(i+1=N-2\).  Hence
\[
  -\rho a_{i+1}=-b_{i+1}-\rho d_{i+1}c_{i+2}.
\]
Substituting this into System B and using
\[
  \rho(1+2\rho)-1=(1+\rho)(2\rho-1)=(1+\rho)v,
\]
we obtain
\[
  b_i
  =
  vb_{i+1}
  +
  \frac{
    \rho c_{i+1}^2/R-c_ic_{i+1}/R
    +c_{i+1}P_i
    -\rho d_{i+1}\sum_{j=i+2}^{N}c_j
  }{1+\rho}.
\]
After division by \(\rho R\),
\begin{equation}
  B_i
  =
  vB_{i+1}
  +
  \frac{
    \rho c_{i+1}^2/R-c_ic_{i+1}/R
    +c_{i+1}P_i
    -\rho d_{i+1}\sum_{j=i+2}^{N}c_j
  }{\rho R(1+\rho)}.
  \label{eq:lower-B-local-step}
\end{equation}

Now put
\[
  X=T_i,\qquad Y=T_{i+1},\qquad Z=T_{i+2}.
\]
Then
\[
  P_{i+1}=X-Y,\qquad P_{i+2}=Y-Z,
\]
so
\[
  d_{i+1}=P_{i+2}-P_{i+1}=2Y-X-Z.
\]
Moreover,
\[
  \frac{c_{i+1}}{R}
  =
  P_{i+1}+\rho P_{i+2}
  =
  X-uY-\rho Z,
\]
and
\[
  \frac{c_i}{R}=P_i+\rho P_{i+1}.
\]
By \eqref{eq:lower-tail-c-sum},
\[
  \frac1R\sum_{j=i+2}^{N}c_j=Y+\rho Z.
\]
In the numerator of the local term in
\eqref{eq:lower-B-local-step}, write
\[
  C=X-uY-\rho Z=\frac{c_{i+1}}{R}.
\]
Then
\[
\begin{aligned}
&\rho c_{i+1}^2/R-c_ic_{i+1}/R
+c_{i+1}P_i
-\rho d_{i+1}\sum_{j=i+2}^{N}c_j                         \\
&\qquad =
R\left\{\rho C^2-(P_i+\rho P_{i+1})C+P_iC
-\rho(2Y-X-Z)(Y+\rho Z)\right\}.
\end{aligned}
\]
Since
\[
  -(P_i+\rho P_{i+1})C+P_iC=-\rho P_{i+1}C=-\rho(X-Y)C,
\]
this equals
\[
  \rho R\left\{
    C^2+(Y-X)C-(2Y-X-Z)(Y+\rho Z)
  \right\}.
\]
Furthermore,
\[
  C+Y-X=(X-uY-\rho Z)+Y-X=\rho(Y-Z).
\]
Therefore the local contribution in \eqref{eq:lower-B-local-step} is
\[
\begin{aligned}
&\frac{
    \rho c_{i+1}^2/R-c_ic_{i+1}/R
    +c_{i+1}P_i
    -\rho d_{i+1}\sum_{j=i+2}^{N}c_j
  }{\rho R(1+\rho)}                                      \\
&\qquad =
\frac{
  \rho(X-uY-\rho Z)(Y-Z)
  -(2Y-X-Z)(Y+\rho Z)
}{1+\rho}.
\end{aligned}
\]
Expanding the numerator and using \(u=1-\rho\) and \(v=2\rho-1\),
\[
\begin{aligned}
&\rho(X-uY-\rho Z)(Y-Z)
  -(2Y-X-Z)(Y+\rho Z)                                   \\
&\qquad =
(1+\rho)\{XY+(\rho-2)Y^2-vYZ+\rho Z^2\}.
\end{aligned}
\]
Therefore
\begin{equation}
  B_i
  =
  vB_{i+1}
  +T_iT_{i+1}+(\rho-2)T_{i+1}^2
  -vT_{i+1}T_{i+2}+\rho T_{i+2}^2 .
  \label{eq:lower-B-step}
\end{equation}

Insert the induction hypothesis \eqref{eq:lower-B-tail-normal} for
\(B_{i+1}\) into \eqref{eq:lower-B-step}.  The terms
\(vT_{i+1}T_{i+2}\) and \(-vT_{i+1}T_{i+2}\) cancel, and we get
\[
\begin{aligned}
  B_i
  &=
  T_iT_{i+1}+(\rho-2)T_{i+1}^2
  +\{v(\rho-2)+\rho\}T_{i+2}^2                         \\
  &\qquad
  +2u^2\sum_{\ell=i+3}^{N-1}v^{\ell-i-2}T_\ell^2 .
\end{aligned}
\]
The first new square coefficient is
\[
  v(\rho-2)+\rho=(2\rho-1)(\rho-2)+\rho
  =2(1-\rho)^2=2u^2.
\]
Thus the term with \(T_{i+2}^2\) is precisely the \(\ell=i+2\) term of
the tail sum in \eqref{eq:lower-B-tail-normal}, while the displayed sum
supplies all terms \(\ell\ge i+3\).  This proves
\eqref{eq:lower-B-tail-normal} by backward induction.

We next derive the residual normal forms.  For \(0\le i\le N-2\),
subtracting the \(i\)-th equation of System B from \(\rho\) times the
corresponding equation of System A gives the local A/B bridge
\begin{equation}
  \rho a_i-b_i=\rho d_ic_{i+1}.
  \label{eq:lower-ab-bridge}
\end{equation}
For \(i\le N-3\) this is obtained from the recursive equations, and for
\(i=N-2\) it is obtained from the terminal equations.  Equivalently,
\[
  a_i=\frac{b_i}{\rho}+d_ic_{i+1}.
\]

Let \(1\le i\le N-2\).  In the middle residual,
\[
  \varepsilon_i
  =
  b_{i-1}+a_i+d_i\sum_{j=i+2}^Nc_j-a_{i-1}-b_i-c_iP_{i-1},
\]
where \(P_{i-1}=1+\sum_{k=0}^{i-2}d_k\).  Substitute
\[
  a_i=\frac{b_i}{\rho}+d_ic_{i+1},
  \qquad
  a_{i-1}=\frac{b_{i-1}}{\rho}+d_{i-1}c_i .
\]
The \(b\)-terms become
\[
  b_{i-1}+\frac{b_i}{\rho}-\frac{b_{i-1}}{\rho}-b_i
  =
  \frac{1-\rho}{\rho}(b_i-b_{i-1}).
\]
After division by \(R\), this gives
\[
  u(B_i-B_{i-1}).
\]
The two \(d_i\)-terms combine as
\[
  d_ic_{i+1}+d_i\sum_{j=i+2}^Nc_j
  =
  d_i\sum_{j=i+1}^Nc_j.
\]
By \eqref{eq:lower-tail-c-sum},
\[
  \frac1R\sum_{j=i+1}^Nc_j=T_i+\rho T_{i+1}.
\]
Also,
\[
  d_i=P_{i+1}-P_i,
\]
and
\[
  d_{i-1}c_i+c_iP_{i-1}=c_i(P_{i-1}+d_{i-1})=c_iP_i.
\]
Finally,
\[
  P_i=T_{i-1}-T_i,
\]
and
\[
  \frac{c_i}{R}
  =
  P_i+\rho P_{i+1}
  =
  (T_{i-1}-T_i)+\rho(T_i-T_{i+1})
  =
  T_{i-1}-uT_i-\rho T_{i+1}.
\]
Therefore
\[
  \frac{\varepsilon_i}{R}
  =
  u(B_i-B_{i-1})
  +(P_{i+1}-P_i)(T_i+\rho T_{i+1})
  -(T_{i-1}-uT_i-\rho T_{i+1})(T_{i-1}-T_i).
\]

We now substitute \eqref{eq:lower-B-tail-normal}.  First,
\[
\begin{aligned}
B_{i-1}
&=
T_{i-1}T_i+(\rho-2)T_i^2
+2u^2T_{i+1}^2                                      \\
&\qquad
+2u^2\sum_{\ell=i+2}^{N-1}v^{\ell-i-1}T_\ell^2 ,
\end{aligned}
\]
where the sum is empty if \(i=N-2\).  Hence
\[
\begin{aligned}
B_i-B_{i-1}
&=
T_iT_{i+1}-T_{i-1}T_i
-(\rho-2)T_i^2                                      \\
&\quad
+\{(\rho-2)-2u^2\}T_{i+1}^2
+2u^2(1-v)\sum_{\ell=i+2}^{N-1}
  v^{\ell-i-2}T_\ell^2 .
\end{aligned}
\]
Since \(1-v=2u\), the tail part contributes
\[
  4u^4\sum_{\ell=i+2}^{N-1}v^{\ell-i-2}T_\ell^2 .
\]

It remains to collect the local terms.  Put
\[
  U=T_{i-1},\qquad X=T_i,\qquad Y=T_{i+1}.
\]
Then
\[
  P_{i+1}-P_i=(X-Y)-(U-X)=2X-U-Y.
\]
The local polynomial is
\[
\begin{aligned}
&u\{XY-UX-(\rho-2)X^2+[(\rho-2)-2u^2]Y^2\}  \\
&\quad +(2X-U-Y)(X+\rho Y)
-(U-uX-\rho Y)(U-X).
\end{aligned}
\]
Expanding and using \(u=1-\rho\), the \(UX\), \(UY\), and \(XY\) terms
cancel, and the remaining terms are
\[
  -U^2+(\rho^2-2\rho+3)X^2
  +(2\rho^3-7\rho^2+8\rho-4)Y^2.
\]
Since
\[
  2\rho^3-7\rho^2+8\rho-4
  =
  (\rho-2)(2\rho^2-3\rho+2),
\]
we obtain
\begin{equation}
  \frac{\varepsilon_i}{R}
  =
  -T_{i-1}^2+\gamma_1T_i^2+\gamma_2T_{i+1}^2
  +4u^4\sum_{\ell=i+2}^{N-1}v^{\ell-i-2}T_\ell^2 ,
  \label{eq:eps-tail-lower-face}
\end{equation}
where
\[
  \gamma_1=\rho^2-2\rho+3,\qquad
  \gamma_2=(\rho-2)(2\rho^2-3\rho+2).
\]

The first coordinate face uses the same normal form, but the residual has
a slightly different endpoint expression.  If \(d_0=0\), then
\eqref{eq:lower-ab-bridge} gives
\[
  a_0=\frac{b_0}{\rho}.
\]
The residual definition gives
\[
  \varepsilon_0=a_0-b_0-c_0,
\]
because \(d_0\sum_{j=2}^Nc_j=0\).  Also \(P_0=P_1=1\), so
\[
  c_0=R(P_0+\rho P_1)=R(1+\rho).
\]
Therefore
\begin{equation}
  \frac{\varepsilon_0}{R}=uB_0-(1+\rho).
  \label{eq:eps0-tail-face}
\end{equation}

It remains only to record the scalar upper bounds used in the proof of
Lemma~\ref{lem:lower-faces}.  These scalar estimates are the only
analytic inequalities needed for the lower-face barrier; all preceding
identities in this appendix are exact algebraic rewritings of Systems C,
A, B, and D.

For \(0<v<1\) and every real \(x\),
\[
\begin{aligned}
  \sum_{r\ge1}v^{r-1}(x-r)^2
  &=
  x^2\sum_{r\ge1}v^{r-1}
  -2x\sum_{r\ge1}rv^{r-1}
  +\sum_{r\ge1}r^2v^{r-1}                                      \\
  &=
  \frac{x^2}{1-v}-\frac{2x}{(1-v)^2}
  +\frac{1+v}{(1-v)^3}.
\end{aligned}
\]
Since \(v=2\rho-1\) and \(u=1-\rho\), we have \(1-v=2u\).

For the first face, the expression inside braces in the main proof is
\[
\begin{aligned}
&x(x+1)+(\rho-2)x^2
+2u^2\sum_{r\ge1}v^{r-1}(x-r)^2                         \\
&\quad =
x(x+1)+(\rho-2)x^2
+2u^2\left\{
    \frac{x^2}{2u}-\frac{2x}{(2u)^2}
    +\frac{1+v}{(2u)^3}
  \right\}.
\end{aligned}
\]
The \(x^2\)-coefficient is
\[
  1+(\rho-2)+u=0,
\]
and the \(x\)-coefficient is
\[
  1-1=0.
\]
The constant term is
\[
  2u^2\frac{1+v}{(2u)^3}
  =
  \frac{1+v}{4u}
  =
  \frac{\rho}{2u}.
\]
Thus
\[
  x(x+1)+(\rho-2)x^2
  +2u^2\sum_{r\ge1}v^{r-1}(x-r)^2
  =
  \frac{\rho}{2u},
\]
and hence
\begin{equation}
\begin{aligned}
&u\left\{
  x(x+1)+(\rho-2)x^2
  +2u^2\sum_{r\ge1}v^{r-1}(x-r)^2
\right\}-(1+\rho)                         \\
&\qquad =
\frac{\rho}{2}-(1+\rho)
=
-\frac{2+\rho}{2}.
\end{aligned}
\label{eq:lower-first-face-scalar}
\end{equation}

For the middle faces, the corresponding expression is
\[
\begin{aligned}
&-(x+2)^2+\gamma_1(x+1)^2+\gamma_2x^2
+4u^4\sum_{r\ge1}v^{r-1}(x-r)^2                         \\
&\quad =
-(x+2)^2+\gamma_1(x+1)^2+\gamma_2x^2                    \\
&\qquad
+4u^4\left\{
    \frac{x^2}{2u}-\frac{2x}{(2u)^2}
    +\frac{1+v}{(2u)^3}
  \right\}                                               \\
&\quad =
  \bigl[-1+\gamma_1+\gamma_2+2u^3\bigr]x^2
  +\bigl[-4+2\gamma_1-2u^2\bigr]x
  +\bigl[-4+\gamma_1+u(1+v)/2\bigr].
\end{aligned}
\]
Using
\[
  \gamma_1=\rho^2-2\rho+3,\qquad
  \gamma_2=(\rho-2)(2\rho^2-3\rho+2),\qquad
  u=1-\rho,\quad v=2\rho-1,
\]
the three bracketed coefficients are
\[
  -1+\gamma_1+\gamma_2+2u^3=0,
\]
\[
  -4+2\gamma_1-2u^2=0,
\]
and
\[
  -4+\gamma_1+\frac{u(1+v)}{2}=-(1+\rho).
\]
Therefore
\begin{equation}
\begin{aligned}
&-(x+2)^2+\gamma_1(x+1)^2+\gamma_2x^2
+4u^4\sum_{r\ge1}v^{r-1}(x-r)^2       \\
&\qquad =-(1+\rho).
\end{aligned}
\label{eq:lower-middle-face-scalar}
\end{equation}
Both scalar right-hand sides are strictly negative.

\section{Algebraic identities for terminal completion}
\label{app:terminal}

This appendix supplies the coefficient-level algebra behind
Lemma~\ref{lem:terminal-quadratic}.  Its role is deliberately narrow: it
proves the terminal quadratic identity used in Section~\ref{sec:terminal}.
It does not use the reduced equations \(E_N(d)=0\), the terminal conclusion
\(L_N(d)=0\), or the later positivity argument.

The direct terminal residual reduction is already proved in
Lemmas~\ref{lem:right-terminal-identity} and
\ref{lem:weighted-left-conservation}.  The present appendix proves only the
remaining weighted quadratic identity
\[
  L_N(d)^2+\Psi_N(\rho)L_N(d)+\Phi_N(\rho)
  =
  \sum_{i=0}^{N-2}\Omega_{N,i}(\rho)\varepsilon_i(d).
\]

\subsection*{Tail variables and target coefficient classes}

Use the same prefix-tail notation as in Appendix~\ref{app:lower-face}:
\[
  P_0=1,\qquad
  P_m=1+\sum_{k=0}^{m-1}d_k\quad(1\le m\le N-1),
  \qquad
  P_N=\rho^{-N}=q^{-1},
\]
\[
  T_m=\sum_{h=m+1}^NP_h\quad(0\le m\le N-1),
  \qquad T_N=0 .
\]
Then
\[
  P_m=T_{m-1}-T_m\quad(1\le m\le N),\qquad
  T_{N-1}=P_N=\rho^{-N},
\]
and
\begin{equation}
  T_0
  =
  N-1+\rho^{-N}+\sum_{j=0}^{N-2}(N-1-j)d_j.
  \label{eq:terminal-T0-affine}
\end{equation}
Consequently, with \(R=q^2=\rho^{2N}\),
\begin{equation}
  L_N(d)=RT_0+\frac{R-1}{1+\rho}.
  \label{eq:terminal-L-tail}
\end{equation}

For coefficient comparison in the original variables, write
\[
  w_j=N-1-j,\qquad
  L_0=q^2(N-1)+\frac{q^2-1}{1+\rho}+q,
  \qquad
  L_N=L_0+q^2\sum_{j=0}^{N-2}w_jd_j .
\]
The weights in Lemma~\ref{lem:terminal-quadratic} are
\[
  \Omega_{N,i}
  =
  q^2\sum_{m=0}^{2N-4}(-1)^m
  \min\{2N-3-m,\;2(N-1-i)\}\rho^m .
\]
Thus the coefficient identities to be proved are, for \(j<k\),
\[
  [d_jd_k]\sum_{i=0}^{N-2}\Omega_{N,i}\varepsilon_i
  =
  2q^4w_jw_k,
\]
for diagonal terms,
\[
  [d_j^2]\sum_{i=0}^{N-2}\Omega_{N,i}\varepsilon_i
  =
  q^4w_j^2,
\]
for linear terms,
\[
  [d_j]\sum_{i=0}^{N-2}\Omega_{N,i}\varepsilon_i
  =
  q^2w_j(2L_0+\Psi_N),
\]
and for the constant term,
\[
  [1]\sum_{i=0}^{N-2}\Omega_{N,i}\varepsilon_i
  =
  L_0^2+\Psi_NL_0+\Phi_N .
\]

\subsection*{Capped alternating-sum coefficients}

Put \(S_0=0\) and, for \(1\le h\le N-1\),
\begin{equation}
  S_h=\sum_{m=0}^{2N-4}(-1)^m
  \min\{2N-3-m,\;2h\}\rho^m ,
  \qquad \Omega_{N,i}=q^2S_{N-1-i}=RS_{N-1-i}.
  \label{eq:terminal-S-def}
\end{equation}
The cap \(2h\) is active when \(m\le2N-3-2h\).  Therefore
\begin{equation}
  S_h=
  2h\sum_{m=0}^{2N-3-2h}(-\rho)^m
  +\sum_{r=1}^{2h-1}r(-\rho)^{2N-3-r},
  \label{eq:capped-split}
\end{equation}
where the first sum is interpreted as empty if \(2N-3-2h<0\).  Indeed, in
the non-capped range set \(r=2N-3-m\), so \(r\) runs from \(2h-1\) down to
\(1\).

\begin{lemma}[Capped first difference]
\label{lem:capped-first-difference}
For \(1\le h\le N-1\),
\[
  S_h-S_{h-1}
  =
  2\sum_{m=0}^{2N-3-2h}(-\rho)^m+\rho^{2N-2-2h}
  =
  \frac{2-(1-\rho)\rho^{2N-2-2h}}{1+\rho}.
\]
\end{lemma}

\begin{proof}
Subtract \eqref{eq:capped-split} at \(h-1\) from
\eqref{eq:capped-split} at \(h\).  For
\(0\le m\le2N-3-2h\), the cap increases from \(2h-2\) to \(2h\), producing
the term
\[
  2\sum_{m=0}^{2N-3-2h}(-\rho)^m.
\]
There is one boundary term at \(m=2N-2-2h\), where the new contribution is
\[
  (-\rho)^{2N-2-2h}=\rho^{2N-2-2h}.
\]
Thus
\[
  S_h-S_{h-1}
  =
  2\sum_{m=0}^{2N-3-2h}(-\rho)^m+\rho^{2N-2-2h}.
\]
The finite geometric sum gives
\[
  2\frac{1-(-\rho)^{2N-2-2h}}{1+\rho}
  +\rho^{2N-2-2h}
  =
  \frac{2-(1-\rho)\rho^{2N-2-2h}}{1+\rho}.
\]
The endpoint \(h=N-1\) is included: the first sum is empty and the formula
gives \(S_{N-1}-S_{N-2}=1\).
\end{proof}

Let
\[
  u=1-\rho,\qquad v=2\rho-1,
\]
\[
  \eta=\rho^2-2\rho+2,\qquad
  \gamma_1=\rho^2-2\rho+3,\qquad
  \gamma_2=(\rho-2)(2\rho^2-3\rho+2).
\]

\subsection*{Reduced residuals in tail form}

The tail normal form for \(b_i\), proved in
Appendix~\ref{app:lower-face}, is
\begin{equation}
  B_i:=\frac{b_i}{\rho R}
  =
  T_iT_{i+1}+(\rho-2)T_{i+1}^2
  +2u^2\sum_{\ell=i+2}^{N-1}v^{\ell-i-2}T_\ell^2 ,
  \qquad 0\le i\le N-2.
  \label{eq:terminal-B-tail-normal}
\end{equation}
This identity is purely algebraic and does not use any coordinate-face
assumption.

The reduced residuals have the following tail forms:
\begin{align}
  \frac{\varepsilon_0}{R}
  &=
  T_0^2-(1+\rho)T_0-1-\eta T_1^2
  +2u^3\sum_{\ell=2}^{N-1}v^{\ell-2}T_\ell^2,
  \label{eq:terminal-eps0-tail}\\
  \frac{\varepsilon_i}{R}
  &=
  -T_{i-1}^2+\gamma_1T_i^2+\gamma_2T_{i+1}^2
  +4u^4\sum_{\ell=i+2}^{N-1}v^{\ell-i-2}T_\ell^2,
  \qquad 1\le i\le N-2 .
  \label{eq:terminal-epsi-tail}
\end{align}
For \(i\ge1\), \eqref{eq:terminal-epsi-tail} is the middle-residual normal
form proved in Appendix~\ref{app:lower-face}.  For \(i=0\), the derivation is
short and is recorded here.  The A/B bridge gives
\[
  a_0=\frac{b_0}{\rho}+d_0c_1 .
\]
Since
\[
  d_0=P_1-P_0=T_0-T_1-1,\qquad
  \frac1R\sum_{j=1}^Nc_j=T_0+\rho T_1,\qquad
  \frac{c_0}{R}=P_0+\rho P_1=1+\rho(T_0-T_1),
\]
the residual definition gives
\[
  \frac{\varepsilon_0}{R}
  =
  uB_0+(T_0-T_1-1)(T_0+\rho T_1)-1-\rho(T_0-T_1).
\]
Substitute \eqref{eq:terminal-B-tail-normal} with \(i=0\).  The mixed
coefficient of \(T_0T_1\) is
\[
  u+\rho-1=0,
\]
the linear \(T_1\)-terms cancel, and the \(T_1^2\)-coefficient is
\[
  -\rho+u(\rho-2)=-\rho^2+2\rho-2=-\eta.
\]
This proves \eqref{eq:terminal-eps0-tail}.

\subsection*{Interior tail-square cancellation}

Multiply \eqref{eq:terminal-eps0-tail} and
\eqref{eq:terminal-epsi-tail} by \(S_{N-1-i}\) and sum over
\(i=0,\ldots,N-2\).  Let
\[
  D_h=S_h-S_{h-1}\qquad(1\le h\le N-1).
\]
By Lemma~\ref{lem:capped-first-difference},
\begin{equation}
  D_h=\frac{2-u\rho^{2N-2-2h}}{1+\rho}.
  \label{eq:D-first-diff}
\end{equation}

The coefficient of \(T_0^2\) is
\[
  S_{N-1}-S_{N-2}=D_{N-1}=1.
\]
The coefficient of \(T_1^2\) is
\[
  -\eta S_{N-1}+\gamma_1S_{N-2}-S_{N-3}.
\]
Since \(\gamma_1=\eta+1\), this becomes
\[
  -\eta(S_{N-1}-S_{N-2})+(S_{N-2}-S_{N-3})
  =
  -\eta D_{N-1}+D_{N-2}.
\]
Using \eqref{eq:D-first-diff},
\[
  D_{N-1}=1,\qquad
  D_{N-2}=\frac{2-u\rho^2}{1+\rho}
  =\rho^2-2\rho+2=\eta,
\]
so the coefficient of \(T_1^2\) is zero.

For \(2\le\ell\le N-2\), the coefficient of \(T_\ell^2\) is
\[
\begin{aligned}
C_\ell
&=
2u^3v^{\ell-2}S_{N-1}
+4u^4\sum_{i=1}^{\ell-2}v^{\ell-i-2}S_{N-1-i}        \\
&\qquad
+\gamma_2S_{N-\ell}+\gamma_1S_{N-1-\ell}
-S_{N-2-\ell}.
\end{aligned}
\]
We now rewrite this coefficient in first differences.  First use
\[
  4u^4=2u^3(1-v),\qquad
  \gamma_1=\eta+1,\qquad
  \gamma_2=-\eta-2u^3 .
\]
The part not multiplied by \(2u^3\) is
\[
\begin{aligned}
&-\eta S_{N-\ell}+(\eta+1)S_{N-1-\ell}-S_{N-2-\ell}  \\
&\qquad =
  (S_{N-1-\ell}-S_{N-2-\ell})
  -\eta(S_{N-\ell}-S_{N-1-\ell})                     \\
&\qquad =
  D_{N-\ell-1}-\eta D_{N-\ell}.
\end{aligned}
\]
The remaining \(2u^3\)-part is
\[
\begin{aligned}
&v^{\ell-2}S_{N-1}
 +(1-v)\sum_{i=1}^{\ell-2}v^{\ell-i-2}S_{N-1-i}
 -S_{N-\ell}                                        \\
&\quad =
 -S_{N-\ell}
 +(1-v)\sum_{s=1}^{\ell-2}v^{s-1}S_{N-\ell+s}
 +v^{\ell-2}S_{N-1}                                  \\
&\quad =
 \sum_{s=0}^{\ell-2}v^s
 (S_{N-\ell+1+s}-S_{N-\ell+s})                       \\
&\quad =
 \sum_{s=0}^{\ell-2}v^sD_{N-\ell+1+s}.
\end{aligned}
\]
Here the second line is just the change of variables \(s=\ell-1-i\).
Consequently
\begin{equation}
  C_\ell
  =
  D_{N-\ell-1}
  -\eta D_{N-\ell}
  +2u^3\sum_{s=0}^{\ell-2}v^sD_{N-\ell+1+s}.
  \label{eq:terminal-tail-cancel-D}
\end{equation}

Now substitute \eqref{eq:D-first-diff} into
\eqref{eq:terminal-tail-cancel-D}.  The numerator of
\((1+\rho)C_\ell\) is
\[
\begin{aligned}
&2-u\rho^{2\ell}
-\eta(2-u\rho^{2\ell-2})
+2u^3\sum_{s=0}^{\ell-2}v^s
       \bigl(2-u\rho^{2\ell-4-2s}\bigr)                 \\
&=
\underbrace{2-2\eta+4u^3\sum_{s=0}^{\ell-2}v^s}_{\mathcal C_\ell}
+
\underbrace{
u\rho^{2\ell-2}(\eta-\rho^2)
-2u^4\sum_{s=0}^{\ell-2}v^s\rho^{2\ell-4-2s}}_{\mathcal P_\ell}.
\end{aligned}
\]
The first finite sum gives
\[
  \sum_{s=0}^{\ell-2}v^s
  =
  \frac{1-v^{\ell-1}}{1-v}
  =
  \frac{1-v^{\ell-1}}{2u}.
\]
Since \(1-\eta=-u^2\),
\[
  \mathcal C_\ell=-2u^2v^{\ell-1}.
\]
For the second finite sum,
\[
\begin{aligned}
  \sum_{s=0}^{\ell-2}v^s\rho^{2\ell-4-2s}
  &=
  \rho^{2\ell-4}
  \sum_{s=0}^{\ell-2}\left(\frac v{\rho^2}\right)^s        \\
  &=
  \frac{\rho^{2\ell-2}-v^{\ell-1}}{u^2},
\end{aligned}
\]
because \(\rho^2-v=(1-\rho)^2=u^2\).  Also
\(\eta-\rho^2=2u\).  Hence
\[
  \mathcal P_\ell
  =
  2u^2\rho^{2\ell-2}
  -2u^2(\rho^{2\ell-2}-v^{\ell-1})
  =
  2u^2v^{\ell-1}.
\]
Therefore \(\mathcal C_\ell+\mathcal P_\ell=0\), so
\(C_\ell=0\).

\begin{lemma}[Interior tail-square cancellation]
\label{lem:interior-tail-square-cancellation}
In
\[
  \sum_{i=0}^{N-2}S_{N-1-i}\frac{\varepsilon_i}{R},
\]
the coefficient of \(T_1^2\) is zero, and the coefficient of
\(T_\ell^2\) is zero for every \(2\le\ell\le N-2\).
\end{lemma}

\begin{proof}
The coefficient of \(T_1^2\) is
\(-\eta D_{N-1}+D_{N-2}=0\).  For \(2\le\ell\le N-2\), the coefficient is
\(C_\ell\), and the computation above proves \(C_\ell=0\).
\end{proof}

\subsection*{Endpoint and constant terms}

The endpoint \(T_{N-1}=\rho^{-N}\) is fixed, so after the interior
cancellations it contributes only to the constant term.  Define
\[
\begin{aligned}
  C_N
  &:=-S_{N-1}
  +\rho^{-2N}\Bigl(
      2u^3v^{N-3}S_{N-1}
      +4u^4\sum_{i=1}^{N-3}v^{N-i-3}S_{N-1-i}
      +\gamma_2S_1
    \Bigr).
\end{aligned}
\]
Then
\begin{equation}
  \sum_{i=0}^{N-2}\Omega_{N,i}\varepsilon_i
  =
  R^2\{T_0^2-(1+\rho)S_{N-1}T_0+C_N\}.
  \label{eq:terminal-weighted-tail-polynomial}
\end{equation}

It remains to match the \(T_0\)-coefficient and the constant coefficient
with the expansion of
\[
  L_N^2+\Psi_NL_N+\Phi_N
  =
  \left(RT_0+\frac{R-1}{1+\rho}\right)^2
  +\Psi_N\left(RT_0+\frac{R-1}{1+\rho}\right)+\Phi_N .
\]
Thus it is enough to prove
\begin{equation}
  2\frac{R-1}{1+\rho}+\Psi_N=-R(1+\rho)S_{N-1},
  \label{eq:terminal-T0-coeff-match}
\end{equation}
and
\begin{equation}
  \left(\frac{R-1}{1+\rho}\right)^2
  +\Psi_N\frac{R-1}{1+\rho}
  +\Phi_N
  =
  R^2C_N .
  \label{eq:terminal-constant-coeff-match}
\end{equation}

Put \(t=-\rho\), and define
\[
  G_M=\sum_{m=0}^Mt^m=\frac{1-t^{M+1}}{1-t},
\]
\[
  H_M=\sum_{m=0}^M(m+1)t^m
  =
  \frac{1-(M+2)t^{M+1}+(M+1)t^{M+2}}{(1-t)^2}.
\]
Then
\[
  S_{N-1}=(2N-2)G_{2N-4}-H_{2N-4},
  \qquad
  \Phi_N=H_{2N-1},
\]
and
\[
  \Psi_N=2G_{2N-1}-(2N-3)R-\rho^{2N+1}G_{2N-4}.
\]

For \eqref{eq:terminal-T0-coeff-match}, substitute the displayed forms of
\(\Psi_N\) and \(S_{N-1}\).  We must show that
\[
  2\frac{R-1}{1+\rho}+\Psi_N+R(1+\rho)S_{N-1}=0.
\]
Using
\[
  S_{N-1}=(2N-2)G_{2N-4}-H_{2N-4},
\]
the left-hand side is
\[
\begin{aligned}
&2\frac{R-1}{1+\rho}
+2G_{2N-1}-(2N-3)R-\rho^{2N+1}G_{2N-4}             \\
&\qquad
+R(1+\rho)\bigl((2N-2)G_{2N-4}-H_{2N-4}\bigr).
\end{aligned}
\]
The closed forms are
\[
  G_{2N-1}=\frac{1-R}{1+\rho},
  \qquad
  G_{2N-4}=\frac{1+\rho^{2N-3}}{1+\rho},
\]
and
\[
  H_{2N-4}
  =
  \frac{1+(2N-2)\rho^{2N-3}+(2N-3)\rho^{2N-2}}
       {(1+\rho)^2}.
\]
After multiplication by \((1+\rho)^2\), the first two terms cancel:
\[
  2(R-1)(1+\rho)+2(1-R)(1+\rho)=0.
\]
The remaining numerator is
\[
\begin{aligned}
&-(2N-3)R(1+\rho)^2
-\rho^{2N+1}(1+\rho)(1+\rho^{2N-3})\\
&\quad
+R(1+\rho)
 \bigl((2N-3)+(2N-2)\rho+\rho^{2N-2}\bigr).
\end{aligned}
\]
Since \(\rho^{2N+1}=R\rho\), this equals
\[
\begin{aligned}
&R(1+\rho)
\Bigl[
  -(2N-3)(1+\rho)
  -\rho(1+\rho^{2N-3})                 \\
&\qquad\qquad
  +(2N-3)+(2N-2)\rho+\rho^{2N-2}
\Bigr].
\end{aligned}
\]
The bracket is
\[
\begin{aligned}
&-(2N-3)-(2N-3)\rho-\rho-\rho^{2N-2}  \\
&\qquad +(2N-3)+(2N-2)\rho+\rho^{2N-2}=0.
\end{aligned}
\]
This proves \eqref{eq:terminal-T0-coeff-match}.

For the constant coefficient, first reverse the finite double sum in
\(C_N\):
\[
\begin{aligned}
\sum_{i=1}^{N-3}v^{N-i-3}S_{N-1-i}
&=
\sum_{m=0}^{2N-4}(-\rho)^m
\sum_{i=1}^{N-3}v^{N-i-3}
   \min\{2N-3-m,\;2N-2-2i\}.
\end{aligned}
\]
For fixed \(m\), define
\[
  I_m:=
  \sum_{i=1}^{N-3}v^{N-i-3}
   \min\{2N-3-m,\;2N-2-2i\},
\]
and set
\[
  r_m=\min\{N-3,\lfloor m/2\rfloor\}.
\]
The split point is determined by
\[
  2N-3-m\le 2N-2-2i
  \quad\Longleftrightarrow\quad
  i\le \frac m2 .
\]
Hence, with empty sums interpreted as zero,
\[
\begin{aligned}
I_m
&=(2N-3-m)\sum_{i=1}^{r_m}v^{N-i-3}
  +\sum_{i=r_m+1}^{N-3}(2N-2-2i)v^{N-i-3}.
\end{aligned}
\]
Let
\[
  \mathcal D_N
  =
  \sum_{m=0}^{2N-4}(-\rho)^m I_m .
\]
Equivalently,
\begin{equation}
  \mathcal D_N=\sum_{i=1}^{N-3}v^{N-i-3}S_{N-1-i}
  =
  \sum_{h=2}^{N-2}v^{h-2}S_h,
  \label{eq:terminal-DN-h-form}
\end{equation}
where the last equality uses \(h=N-1-i\).

After substitution of the displayed formulas for \(\Psi_N,\Phi_N,S_{N-1}\),
\(S_1\), and \(\mathcal D_N\), the numerator of
\eqref{eq:terminal-constant-coeff-match} becomes
\[
\begin{aligned}
\mathcal N_N(\rho)
&=
 (R-1)^2
 +(1+\rho)(R-1)\Psi_N
 +(1+\rho)^2\Phi_N
 +(1+\rho)^2R^2S_{N-1}                              \\
&\quad
 -(1+\rho)^2R
 \Bigl(
      2u^3v^{N-3}S_{N-1}
      +4u^4\mathcal D_N
      +\gamma_2S_1
 \Bigr).
\end{aligned}
\]
Split this numerator into three groups:
\[
\begin{aligned}
  \mathcal N_N^{(0)}
  &:=(R-1)^2
    +(1+\rho)(R-1)\Psi_N
    +(1+\rho)^2\Phi_N,\\
  \mathcal N_N^{(1)}
  &:=(1+\rho)^2R^2S_{N-1}
    -(1+\rho)^2R\bigl(2u^3v^{N-3}S_{N-1}+\gamma_2S_1\bigr),\\
  \mathcal N_N^{(2)}
  &:=-4(1+\rho)^2Ru^4\mathcal D_N .
\end{aligned}
\]

\begin{lemma}[Terminal constant cancellation]
\label{lem:terminal-constant-cancellation}
The three constant-term groups satisfy
\[
  \mathcal N_N^{(0)}+\mathcal N_N^{(1)}
  =
  4(1+\rho)^2Ru^4\mathcal D_N,
  \qquad
  \mathcal N_N^{(2)}
  =
  -4(1+\rho)^2Ru^4\mathcal D_N.
\]
Consequently \(\mathcal N_N=0\).
\end{lemma}

\begin{proof}
The second identity is the definition of \(\mathcal N_N^{(2)}\).  We prove
the first identity.

The finite identities
\[
  (1+\rho)G_M=1-(-\rho)^{M+1},
  \qquad
  (1+\rho)^2H_M
  =
  1-(M+2)(-\rho)^{M+1}+(M+1)(-\rho)^{M+2}
\]
give, for the three sources in \(\mathcal N_N^{(0)}\),
\[
\resizebox{\textwidth}{!}{$
\begin{array}{c|c|c}
\text{source}
& \text{finite-sum form}
& \text{reduced contribution}\\ \hline
(R-1)^2
& (R-1)^2
& R^2-2R+1\\[1mm]
(1+\rho)(R-1)\Psi_N
& (1+\rho)(R-1)
  \left(2G_{2N-1}-(2N-3)R-\rho^{2N+1}G_{2N-4}\right)
& (R-1)\left(2(1-R)-(2N-3)(1+\rho)R-\rho^{2N+1}(1+\rho^{2N-3})\right)\\[1mm]
(1+\rho)^2\Phi_N
& (1+\rho)^2H_{2N-1}
& 1-(2N+1)R-2N\rho R .
\end{array}
$}
\]
Adding the three reduced contributions gives the explicit numerator
\[
\begin{aligned}
\mathcal N_N^{(0)}
&=
R^2-2R+1                                                   \\
&\quad
+(R-1)\left(
  2(1-R)-(2N-3)(1+\rho)R-\rho^{2N+1}(1+\rho^{2N-3})
\right)                                                     \\
&\quad
+1-(2N+1)R-2N\rho R .
\end{aligned}
\]
We next rewrite this expression in terms of \(S_{N-1}\) and \(S_1\).  The
closed forms used are
\[
  (1+\rho)^2S_{N-1}
  =
  (2N-3)+(2N-2)\rho+\rho^{2N-2},
\]
and, since \(S_1=2\sum_{m=0}^{2N-5}(-\rho)^m+\rho^{2N-4}\),
\[
\begin{aligned}
  (1+\rho)^2S_1
  &=
  2(1+\rho)(1-\rho^{2N-4})+(1+\rho)^2\rho^{2N-4}.
\end{aligned}
\]
Substituting these two formulas into
\[
 -(1+\rho)^2R\left(RS_{N-1}+S_1+u^2\rho^{2N-4}\right)
\]
and using \(R=\rho^{2N}\) gives exactly the preceding expanded numerator.
Thus
\begin{equation}
  \mathcal N_N^{(0)}
  =
  -(1+\rho)^2R\left(RS_{N-1}+S_1+u^2\rho^{2N-4}\right).
  \label{eq:N0-source-separated}
\end{equation}
The second group is already in capped-sum form:
\begin{equation}
  \mathcal N_N^{(1)}
  =
  (1+\rho)^2R
  \left(RS_{N-1}-2u^3v^{N-3}S_{N-1}-\gamma_2S_1\right).
  \label{eq:N1-source-separated}
\end{equation}
Adding \eqref{eq:N0-source-separated} and
\eqref{eq:N1-source-separated} cancels the two \(RS_{N-1}\)-terms.  Using
\[
  1+\gamma_2=u^2(v-2),\qquad 1-v=2u,
\]
we get
\begin{equation}
\begin{aligned}
  \mathcal N_N^{(0)}+\mathcal N_N^{(1)}
  &=(1+\rho)^2Ru^2
  \left((2-v)S_1-\rho^{2N-4}-(1-v)v^{N-3}S_{N-1}\right).
\end{aligned}
\label{eq:N01-after-source-separation}
\end{equation}

It remains to identify the bracket in
\eqref{eq:N01-after-source-separation}.  From
\eqref{eq:terminal-DN-h-form},
\[
  \mathcal D_N=\sum_{h=2}^{N-2}v^{h-2}S_h .
\]
If \(N=3\), this sum is empty and the identity below reduces to
\[
  (2-v)S_1-\rho^2-(1-v)S_2=0.
\]
This endpoint case is still a finite-sum identity, not an exception:
\[
  S_1=2-2\rho+\rho^2,\qquad S_2=3-2\rho+\rho^2,
\]
and therefore
\[
\begin{aligned}
 &(2-v)S_1-\rho^2-(1-v)S_2                         \\
 &=(3-2\rho)(2-2\rho+\rho^2)-\rho^2
   -2(1-\rho)(3-2\rho+\rho^2)=0 .
\end{aligned}
\]
Assume now \(N\ge4\).  With
\(\Delta_h=S_h-S_{h-1}\), summation by parts gives
\[
  (1-v)\mathcal D_N
  =
  S_2+\sum_{h=3}^{N-2}v^{h-2}\Delta_h
  -v^{N-3}S_{N-2}.
\]
Adding \(v^{N-3}S_{N-1}\) to both sides gives
\[
  (1-v)\mathcal D_N+v^{N-3}S_{N-1}
  =
  S_1+\sum_{h=2}^{N-1}v^{h-2}\Delta_h .
\]
Therefore it is enough to evaluate
\[
  (1-v)\sum_{h=2}^{N-1}v^{h-2}\Delta_h .
\]
Using Lemma~\ref{lem:capped-first-difference} and setting \(t=h-2\),
\[
\begin{aligned}
&(1-v)\sum_{h=2}^{N-1}v^{h-2}\Delta_h                                      \\
&=
\frac{1-v}{1+\rho}
\left(
2\sum_{t=0}^{N-3}v^t
-u\sum_{t=0}^{N-3}v^t\rho^{2N-6-2t}
\right).
\end{aligned}
\]
The two finite sums are
\[
  \sum_{t=0}^{N-3}v^t=\frac{1-v^{N-2}}{1-v}
\]
and
\[
  \sum_{t=0}^{N-3}v^t\rho^{2N-6-2t}
  =
  \frac{\rho^{2N-4}-v^{N-2}}{u^2},
  \qquad \rho^2-v=u^2.
\]
Since \(1-v=2u\), the preceding expression reduces to
\[
  \frac{2(1-\rho^{2N-4})}{1+\rho}.
\]
But
\[
  S_1-\rho^{2N-4}
  =
  2\sum_{m=0}^{2N-5}(-\rho)^m
  =
  \frac{2(1-\rho^{2N-4})}{1+\rho}.
\]
Thus
\[
  (1-v)\sum_{h=2}^{N-1}v^{h-2}\Delta_h=S_1-\rho^{2N-4}.
\]
Combining the last displays gives
\[
  (1-v)^2\mathcal D_N
  =
  (2-v)S_1-\rho^{2N-4}-(1-v)v^{N-3}S_{N-1}.
\]
Since \(1-v=2u\),
\[
  (2-v)S_1-\rho^{2N-4}-(1-v)v^{N-3}S_{N-1}
  =
  4u^2\mathcal D_N .
\]
Substitution into \eqref{eq:N01-after-source-separation} proves
\[
  \mathcal N_N^{(0)}+\mathcal N_N^{(1)}
  =
  4(1+\rho)^2Ru^4\mathcal D_N .
\]
Together with the definition of \(\mathcal N_N^{(2)}\), this proves
\(\mathcal N_N=0\).
\end{proof}

Lemma~\ref{lem:terminal-constant-cancellation} proves
\eqref{eq:terminal-constant-coeff-match}.  Hence
\eqref{eq:terminal-weighted-tail-polynomial} matches
\(L_N^2+\Psi_NL_N+\Phi_N\) in the \(T_0^2\), \(T_0\), and constant
coefficients.

\begin{lemma}[Endpoint and constant comparison]
\label{lem:endpoint-constant-comparison}
After the interior tail-square cancellation,
\[
  \sum_{i=0}^{N-2}\Omega_{N,i}\varepsilon_i
  =
  R^2\{T_0^2-(1+\rho)S_{N-1}T_0+C_N\},
\]
and the \(T_0\)- and constant coefficients of this expression are exactly
the corresponding coefficients of
\(L_N^2+\Psi_NL_N+\Phi_N\).
\end{lemma}

\begin{proof}
The identity \eqref{eq:terminal-weighted-tail-polynomial} gives the weighted
tail polynomial.  The \(T_0\)-coefficient is matched by
\eqref{eq:terminal-T0-coeff-match}, and the constant coefficient is matched
by \eqref{eq:terminal-constant-coeff-match}, which follows from
Lemma~\ref{lem:terminal-constant-cancellation}.
\end{proof}

\begin{lemma}[Terminal coefficient classes]
\label{lem:terminal-coefficient-classes}
The weighted residual identity
\[
  \sum_{i=0}^{N-2}\Omega_{N,i}\varepsilon_i
  =
  L_N^2+\Psi_NL_N+\Phi_N
\]
holds as a polynomial identity in \(d_0,\ldots,d_{N-2}\).
\end{lemma}

\begin{proof}
By Lemma~\ref{lem:interior-tail-square-cancellation} and
Lemma~\ref{lem:endpoint-constant-comparison}, the two sides agree as
polynomials in the affine variable \(T_0\).  Now use
\eqref{eq:terminal-T0-affine}.  Since
\[
  T_0=N-1+\rho^{-N}+\sum_{j=0}^{N-2}w_jd_j,\qquad w_j=N-1-j,
\]
and \(L_N=L_0+q^2\sum_jw_jd_j\), the coefficients of
\(L_N^2+\Psi_NL_N+\Phi_N\) are
\[
  [d_jd_k]=2q^4w_jw_k\quad(j<k),\qquad
  [d_j^2]=q^4w_j^2,
\]
\[
  [d_j]=q^2w_j(2L_0+\Psi_N),
  \qquad
  [1]=L_0^2+\Psi_NL_0+\Phi_N .
\]
These are precisely the four target coefficient classes listed at the
start of the appendix.  Since both sides have degree at most two in
\(d_0,\ldots,d_{N-2}\), the coefficient check exhausts all monomials.
\end{proof}

\section{Algebraic identities for the tail-square recurrence}
\label{app:tail-square}

This appendix proves the algebra used in Lemma~\ref{lem:tail-square}.  The
point is to avoid repeating the long lower-face derivation from
Appendix~\ref{app:lower-face}.  The normal form for \(B_i=b_i/(\rho R)\) and
the middle-residual tail expansion have already been proved there from
Systems C, A, B, and D.  Here we restate those identities in the
terminal-completed notation, derive the terminal residual
\(\varepsilon_{N-1}\), and then turn the residual equations into a scalar
convolution recurrence.

Put
\[
  T_m=\sum_{h=m+1}^NP_h\quad(0\le m\le N-1),\qquad T_N:=0 .
\]
Then
\[
  T_{N-1}=P_N=\rho^{-N},\qquad
  P_m=T_{m-1}-T_m\quad(1\le m\le N).
\]
At a terminal-completed reduced zero, \(L_N(d)=0\) gives
\[
  \sum_{h=1}^{N-1}P_h=2N-P_N.
\]
Consequently
\[
  T_m=\sum_{h=m+1}^NP_h=2N-\sum_{h=1}^mP_h,
  \qquad 0\le m\le N-1.
\]
Thus this appendix uses the same tail variables as
Section~\ref{sec:coeff-pos}; the equality with the abstract tails from
Appendix~\ref{app:lower-face} is a consequence of terminal completion, not a
new endpoint assumption.

System C gives
\[
  c_i=R(P_i+\rho P_{i+1})\quad(0\le i\le N-1),
  \qquad c_N=RP_N.
\]
Hence, for \(1\le k\le N\),
\begin{equation}
  \sum_{j=k}^Nc_j
  =
  R(T_{k-1}+\rho T_k).
  \label{eq:tail-c-sum}
\end{equation}
Indeed,
\[
\begin{aligned}
  \sum_{j=k}^Nc_j
  &=
  R\sum_{j=k}^{N-1}(P_j+\rho P_{j+1})+RP_N          \\
  &=
  R\left\{
    P_k+(1+\rho)\sum_{h=k+1}^{N}P_h
  \right\}                                          \\
  &=
  R\left\{
    (T_{k-1}-T_k)+(1+\rho)T_k
  \right\}
  =
  R(T_{k-1}+\rho T_k).
\end{aligned}
\]
For \(k=N\), the convention \(T_N=0\) gives
\(R(T_{N-1}+\rho T_N)=RP_N=c_N\).

Let
\[
  u=1-\rho,\qquad v=2\rho-1,\qquad
  B_i=\frac{b_i}{\rho R}\quad(0\le i\le N-2).
\]
Appendix~\ref{app:lower-face} proves the tail normal form
\begin{equation}
  B_i
  =
  T_iT_{i+1}+(\rho-2)T_{i+1}^2
  +2u^2\sum_{\ell=i+2}^{N-1}v^{\ell-i-2}T_\ell^2 .
  \label{eq:B-tail-normal}
\end{equation}
This is the same identity as \eqref{eq:lower-B-tail-normal}; it is restated
here only so that the later formulas in this appendix have local labels.

For \(1\le i\le N-2\), the middle-residual normal form from
Appendix~\ref{app:lower-face} is
\begin{equation}
  \frac{\varepsilon_i}{R}
  =
  -T_{i-1}^2+\gamma_1T_i^2+\gamma_2T_{i+1}^2
  +4u^4\sum_{\ell=i+2}^{N-1}v^{\ell-i-2}T_\ell^2 ,
  \label{eq:eps-tail-expanded}
\end{equation}
where
\[
  \gamma_1=\rho^2-2\rho+3,\qquad
  \gamma_2=(\rho-2)(2\rho^2-3\rho+2).
\]
This is the same formula as \eqref{eq:eps-tail-lower-face}.  Its derivation
does not use terminal completion or positivity.

It remains to compute the terminal residual in the same tail variables.  From
\eqref{eq:Dterm1},
\[
  \varepsilon_{N-1}
  =
  b_{N-2}+a_{N-1}-a_{N-2}-c_{N-1}P_{N-2}.
\]
The A/B bridge at \(N-2\) gives
\[
  a_{N-2}=\frac{b_{N-2}}{\rho}+d_{N-2}c_{N-1}.
\]
Since \(d_{N-2}+P_{N-2}=P_{N-1}\), we obtain
\[
  \varepsilon_{N-1}
  =
  -\frac{u}{\rho}b_{N-2}+a_{N-1}-c_{N-1}P_{N-1}.
\]
Put
\[
  X=T_{N-2},\qquad Y=T_{N-1}=P_N.
\]
Then
\[
  P_{N-1}=X-Y,\qquad
  \frac{c_{N-1}}R=P_{N-1}+\rho P_N=X-uY,\qquad
  RY^2=1.
\]
The terminal equation of System A gives
\[
  a_{N-1}=1-c_NP_{N-1}=1-RY(X-Y).
\]
Using \eqref{eq:B-tail-normal} at \(i=N-2\),
\[
  B_{N-2}=XY+(\rho-2)Y^2.
\]
Divide the formula for \(\varepsilon_{N-1}\) by \(R\).  Since
\(1/R=Y^2\),
\[
\begin{aligned}
  \frac{\varepsilon_{N-1}}R
  &=
  -u\{XY+(\rho-2)Y^2\}
  +\{Y^2-Y(X-Y)\}
  -(X-uY)(X-Y)                                      \\
  &=
  -X^2+(\rho^2-2\rho+3)Y^2 .
\end{aligned}
\]
Thus
\begin{equation}
  \frac{\varepsilon_{N-1}}R
  =
  -T_{N-2}^2+\gamma_1T_{N-1}^2 .
  \label{eq:terminal-tail-expanded}
\end{equation}

Now define
\[
  X_n=RT_{N-1-n}^2,\qquad 0\le n\le N-1.
\]
Then
\[
  X_0=RT_{N-1}^2=R\rho^{-2N}=1.
\]
Equation \eqref{eq:terminal-tail-expanded} gives
\[
  \varepsilon_{N-1}=-X_1+\gamma_1.
\]
For \(1\le i\le N-2\), put \(n=N-i\).  Then
\eqref{eq:eps-tail-expanded} becomes
\[
  \varepsilon_i=\sum_{s=0}^{n}\gamma_sX_{n-s},
\]
where
\[
  \gamma_0=-1,\qquad
  \gamma_1=\rho^2-2\rho+3,\qquad
  \gamma_2=(\rho-2)(2\rho^2-3\rho+2),
\]
and, for \(s\ge3\),
\[
  \gamma_s=4(\rho-1)^4(2\rho-1)^{s-3}.
\]

We next verify the comparison sequence.  Let
\[
  A_L(\rho)=\sum_{s=0}^{L}(-1)^s(L+1-s)\rho^s,
  \qquad B_n=A_{2n}(\rho).
\]
Splitting \(A_{2n}\) into even and odd powers gives
\[
  A_{2n}
  =
  \sum_{k=0}^{n}(2n+1-2k)\rho^{2k}
  -\rho\sum_{k=0}^{n-1}(2n-2k)\rho^{2k}.
\]
Therefore
\begin{equation}
\begin{aligned}
  \sum_{n\ge0}B_nz^n
  &=
  \left(\sum_{m\ge0}(2m+1)z^m\right)
  \left(\sum_{k\ge0}(\rho^2z)^k\right)                 \\
  &\quad
  -\rho\left(\sum_{m\ge1}2mz^m\right)
  \left(\sum_{k\ge0}(\rho^2z)^k\right)                 \\
  &=
  \frac{1+z}{(1-z)^2}\frac1{1-\rho^2z}
  -\rho\frac{2z}{(1-z)^2}\frac1{1-\rho^2z}              \\
  &=
  \frac{1+z-2\rho z}{(1-z)^2(1-\rho^2z)} .
\end{aligned}
\label{eq:tail-square-B-generating}
\end{equation}
On the other hand,
\begin{equation}
\begin{aligned}
  \sum_{s\ge0}\gamma_sz^s
  &=
  -1+(\rho^2-2\rho+3)z
  +(\rho-2)(2\rho^2-3\rho+2)z^2                         \\
  &\quad
  +4(1-\rho)^4z^3\sum_{t\ge0}\bigl((2\rho-1)z\bigr)^t  \\
  &=
  -1+(\rho^2-2\rho+3)z
  +(\rho-2)(2\rho^2-3\rho+2)z^2
  +\frac{4(1-\rho)^4z^3}{1-(2\rho-1)z}                  \\
  &=
  -\frac{(1-z)^2(1-\rho^2z)}{1+z-2\rho z}.
\end{aligned}
\label{eq:tail-square-gamma-generating}
\end{equation}
Consequently
\begin{equation}
  \left(\sum_{s\ge0}\gamma_sz^s\right)
  \left(\sum_{n\ge0}B_nz^n\right)
  =
  -1.
  \label{eq:tail-square-generating-product}
\end{equation}
Equating coefficients gives
\begin{equation}
  \sum_{s=0}^{n}\gamma_sA_{2n-2s}=0,
  \qquad n\ge1 .
  \label{eq:tail-square-A-convolution}
\end{equation}

At a nonnegative reduced zero, Lemma~\ref{lem:terminal-completion} gives
\(\varepsilon_{N-1}=0\).  Hence
\[
  X_1=\gamma_1=A_2.
\]
For \(2\le n\le N-1\), the reduced residual equation
\(\varepsilon_{N-n}=0\) gives
\[
  \sum_{s=0}^{n}\gamma_sX_{n-s}=0.
\]
Because \(\gamma_0=-1\), this equation determines \(X_n\) uniquely from
\(X_0,\ldots,X_{n-1}\).  The sequence \(A_{2n}\) satisfies the same recursion
and the same initial values \(A_0=1=X_0\), \(A_2=X_1\).  Therefore
\[
  X_n=A_{2n}(\rho),\qquad 0\le n\le N-1.
\]

\section{Cumulative-margin algebra}
\label{app:cumulative-margin}

This appendix expands the algebra used in the proof of
Theorem~\ref{thm:K-positive}.  Its purpose is to show explicitly how the
cumulative margin \(K_i\) reduces, after the tail-square law, to the scalar
alternating-truncation inequality in Lemma~\ref{lem:scalar-A}.

Fix \(0\le i\le N-2\) and put
\[
  m=i+1.
\]
Recall
\[
  K_i
  =
  \sum_{h=0}^{i}c_h
  -
  \left(\sum_{h=0}^{i}d_h\right)
  \left(\sum_{j=i+1}^{N}c_j\right).
\]
Since \(m=i+1\), this is
\[
  K_i
  =
  \sum_{h=0}^{m-1}c_h
  -
  (P_m-1)\sum_{j=m}^{N}c_j,
\]
because
\[
  P_m=1+\sum_{h=0}^{m-1}d_h.
\]

We first rewrite the two \(c\)-sums.  By System C,
\[
\begin{aligned}
  \sum_{h=0}^{m-1}\frac{c_h}{R}
  &=
  \sum_{h=0}^{m-1}(P_h+\rho P_{h+1})                  \\
  &=
  P_0+\sum_{h=1}^{m-1}P_h+\rho\sum_{h=1}^{m}P_h       \\
  &=
  1+(1+\rho)\Pi_{m-1}+\rho P_m .
\end{aligned}
\]
Here
\[
  \Pi_{m-1}=\sum_{h=1}^{m-1}P_h.
\]
The tail-sum identity gives
\[
  \sum_{j=m}^{N}\frac{c_j}{R}
  =
  T_{m-1}+\rho T_m.
\]
At a terminal-completed reduced zero,
\[
  T_{m-1}=2N-\Pi_{m-1},\qquad
  T_m=T_{m-1}-P_m=2N-\Pi_{m-1}-P_m.
\]
Therefore
\[
\begin{aligned}
  T_{m-1}+\rho T_m
  &=
  (2N-\Pi_{m-1})
  +\rho(2N-\Pi_{m-1}-P_m)       \\
  &=
  (1+\rho)(2N-\Pi_{m-1})-\rho P_m.
\end{aligned}
\]

Substituting the two displayed sums into \(K_i\) gives
\[
\begin{aligned}
  \frac{K_i}{R}
  &=
  1+(1+\rho)\Pi_{m-1}+\rho P_m                    \\
  &\quad
  -(P_m-1)\{(1+\rho)(2N-\Pi_{m-1})-\rho P_m\}.
\end{aligned}
\]
Expanding the product,
\[
\begin{aligned}
  \frac{K_i}{R}
  &=
  1+(1+\rho)\Pi_{m-1}+\rho P_m                    \\
  &\quad
  -P_m(1+\rho)(2N-\Pi_{m-1})
  +\rho P_m^2                                     \\
  &\quad
  +(1+\rho)(2N-\Pi_{m-1})-\rho P_m.
\end{aligned}
\]
The two linear \(\rho P_m\) terms cancel, and the terms
\((1+\rho)\Pi_{m-1}\) and \(-(1+\rho)\Pi_{m-1}\) cancel.  Hence
\begin{equation}
  \frac{K_i}{R}
  =
  1+\rho P_m^2
  +(1+\rho)\{2N-P_m(2N-\Pi_{m-1})\}.
  \label{eq:cumulative-K-prefix-form}
\end{equation}

Now define
\[
  U=T_{m-1}=2N-\Pi_{m-1},\qquad
  V=T_m=U-P_m.
\]
Then
\[
  P_m=U-V.
\]
Using the root equation in the form
\[
  1+2N(1+\rho)=R^{-1},
\]
we obtain
\begin{equation}
\begin{aligned}
  \frac{K_i}{R}
  &=
  1+\rho(U-V)^2+(1+\rho)\{2N-(U-V)U\}             \\
  &=
  1+2N(1+\rho)
  +\rho(U^2-2UV+V^2)
  -(1+\rho)(U^2-UV)                               \\
  &=
  R^{-1}
  -U^2+\rho V^2+(1-\rho)UV .
\end{aligned}
\label{eq:cumulative-K-UV-form}
\end{equation}

Let
\[
  r=N-m.
\]
Since \(0\le i\le N-2\) and \(m=i+1\), we have
\[
  1\le r\le N-1.
\]
The tail-square law gives
\[
  U^2=T_{m-1}^2=R^{-1}A_{2r},
  \qquad
  V^2=T_m^2=R^{-1}A_{2r-2}.
\]
Because the tails \(U,V\) are positive,
\[
  UV=R^{-1}\sqrt{A_{2r}A_{2r-2}}.
\]
Substituting into \eqref{eq:cumulative-K-UV-form}, we get
\begin{equation}
  \frac{K_i}{R}
  =
  R^{-1}
  \left\{
    1-A_{2r}+\rho A_{2r-2}
    +(1-\rho)\sqrt{A_{2r}A_{2r-2}}
  \right\}.
  \label{eq:cumulative-K-tail-square-form}
\end{equation}

It remains to reduce the non-square part.  From
\[
  A_L(\rho)=\sum_{s=0}^{L}(-1)^s(L+1-s)\rho^s,
\]
one checks directly that
\begin{equation}
  A_{2r}-\rho A_{2r-2}-1=(1-\rho)A_{2r-1}.
  \label{eq:A-odd-reduction}
\end{equation}
Indeed, the constant term on both sides is \(2r\).  For
\(1\le s\le 2r-1\), the coefficient of \(\rho^s\) on the left is
\[
  (-1)^s(2r+1-s)-(-1)^{s-1}(2r-s)
  =
  (-1)^s(4r+1-2s),
\]
which is exactly the coefficient obtained from
\((1-\rho)A_{2r-1}\).  The top coefficient of \(\rho^{2r}\) is \(1\) on
both sides.

Using \eqref{eq:A-odd-reduction},
\[
  1-A_{2r}+\rho A_{2r-2}=-(1-\rho)A_{2r-1}.
\]
Therefore
\begin{equation}
  \frac{K_i}{R}
  =
  R^{-1}(1-\rho)
  \left\{
    \sqrt{A_{2r-2}A_{2r}}-A_{2r-1}
  \right\}.
  \label{eq:cumulative-K-final-form}
\end{equation}

Since \(0<\rho<1\), \(R^{-1}(1-\rho)>0\).  Also
\[
  A_{2r}=RU^2>0,\qquad A_{2r-2}=RV^2>0.
\]
Hence \(K_i>0\) follows from
\[
  \sqrt{A_{2r-2}A_{2r}}>A_{2r-1}.
\]
A sufficient condition is
\[
  A_{2r-2}A_{2r}-A_{2r-1}^2>0,
\]
because this implies
\[
  \sqrt{A_{2r-2}A_{2r}}>|A_{2r-1}|\ge A_{2r-1}.
\]
This is precisely Lemma~\ref{lem:scalar-A}.  Thus cumulative positivity is
not a new independent estimate: after the tail-square law, it is reduced to
the scalar alternating-truncation inequality.

\end{document}